\documentclass[11pt,twoside]{article}

\usepackage[square,numbers]{natbib}

\iffalse 
\usepackage[autostyle]{csquotes}
\usepackage[
    backend=biber,
       style=numeric,
    firstinits = true,
    
    isnb=false,
    url=false, 
    doi=true,
    eprint=true
]{biblatex}

\renewbibmacro{in:}{
  \ifentrytype{article}{}{
  \printtext{\bibstring{in}\intitlepunct}}}

\DeclareFieldFormat[article,unpublished,book]{pages}{#1}

\DeclareFieldFormat[article]{title}{#1}
\DeclareFieldFormat[unpublished]{title}{#1}
\DeclareFieldFormat[misc]{title}{#1}
\fi

\usepackage{lipsum}

\usepackage{microtype} 

\usepackage{lmodern}
\usepackage[sans]{dsfont}
\usepackage[latin1]{inputenc}
\usepackage{color}
\usepackage{delimdelim}

\usepackage[hmarginratio=1:1,top=32mm,columnsep=20pt]{geometry} 
\usepackage{booktabs} 

\usepackage{hyperref}

\usepackage{paralist}

\usepackage{abstract}

\newcommand{\ep}{\epsilon}
\newcommand{\m}{\mbox{d}}

\newcommand{\f}{\boldsymbol}

\newcommand{\tildej}{\tilde{\jmath}}
\newcommand{\tildebj}{\tilde{\f{\jmath}}}

\newcommand{\K}{K_{H^s\rightarrow C^0}}
\definecolor{blue1}{rgb}{0,0,0}

\usepackage{titlesec} 

\usepackage{authblk}

\usepackage{fancyhdr} 
\usepackage{amsmath}
\usepackage{amssymb}
\usepackage{amsthm}
\usepackage{subfigure}
\usepackage{textcomp}
\usepackage{listings}
\usepackage{setspace}
\usepackage{graphicx}
\usepackage{epstopdf}
\usepackage{eurosym}
\usepackage{inputenc}
\usepackage{inputenc}
\usepackage{appendix}
\usepackage{tabularx}
\usepackage{array}
\usepackage{makeidx}
\usepackage{etoolbox}
\usepackage{mathtools}
\usepackage{empheq}
\usepackage{mathrsfs}
\usepackage{marginnote,xspace}

\DeclareGraphicsExtensions{.eps}
\DeclareGraphicsExtensions{.png}

\newcommand\iid{{\emph{iid}}\xspace}

\theoremstyle{plain}
\newtheorem{theorem}{Theorem}[section]

\newtheorem{lemma}[theorem]{Lemma}

\newtheorem*{assumg*}{Assumption~(G)}

\newtheorem*{assumng*}{Assumption~(NG)}
\newtheorem*{assumng1*}{Assumption~(NG-bis)}

\newcommand{\per}{\mathrm{per}}

\theoremstyle{definition}

\newtheorem{rem}[theorem]{Remark}
\hypersetup{citecolor=black}

\newcommand*{\doi}[1]{\href{http://dx.doi.org/#1}{doi: #1}}

\newlength\tindent
\title{\vspace{-15mm}\fontsize{24pt}{10pt} \LARGE{Well-posedness for a regularised inertial Dean--Kawasaki model for slender
    particles in several space dimensions}} \small{} \date{}

\begin{document}
\author[1]{Federico Cornalba\footnote{\texttt{federico.cornalba@ist.ac.at}}}
\affil[1]{\normalsize{Institute of Science and Technology Austria}}
\author[2]{Tony Shardlow\footnote{\texttt{T.Shardlow@bath.ac.uk}}}
\author[3]{Johannes Zimmer\footnote{\texttt{jz@ma.tum.de}}}
\affil[2]{University of Bath, United Kingdom}
\affil[3]{Technische Universit\"at M\"unchen, Germany}

\maketitle 
\thispagestyle{empty}

\pagestyle{fancy} 

\renewenvironment{abstract}
 {\small
  \begin{center}
  \bfseries \abstractname\vspace{-0.0 pc}\vspace{0pt}
  \end{center}
  \list{}{
    \setlength{\leftmargin}{10mm}
    \setlength{\rightmargin}{\leftmargin}
  }
  \item\relax}
 {\endlist}

 \newenvironment{thanks}
 {  \list{}{
    \setlength{\leftmargin}{0mm}
    \setlength{\rightmargin}{\leftmargin}
  }
  \item\relax}
 {\endlist}

\vspace{-3 pc} 

\begin{abstract}
  A stochastic PDE, describing mesoscopic fluctuations in systems of weakly interacting inertial particles of finite volume, is
  proposed and analysed in any finite dimension $d\in\mathbb{N}$. It is a regularised and inertial version of the Dean--Kawasaki
  model. A high-probability well-posedness theory for this model is developed.  This theory improves significantly on the spatial
  scaling restrictions imposed in an earlier work of the same authors, which applied only to significantly larger particles in one
  dimension. The well-posedness theory now applies in $d$-dimensions when the particle-width $\ep$ is proportional to
  $N^{-1/\theta}$ for $\theta>2d$ and $N$ is the number of particles. This scaling is optimal in a certain Sobolev norm. Key tools
  of the analysis are fractional Sobolev spaces, sharp bounds on Bessel functions, separability of the regularisation in the
  $d$-spatial dimensions, and use of the Fa\`a di Bruno's formula.

  
\end{abstract}

\section{Introduction}

Fluctuating hydrodynamics is a class of models describing fluctuations around the hydrodynamic limit of a many-particle system; a
particular example is the Dean--Kawasaki model~\cite{Dean1996a,Kawasaki1998a}, which describes the evolution of finitely many
particles governed by over-damped Langevin dynamics.  At its core, this model is a stochastic PDE for the empirical density,
comprising a diffusion equation that is stochastically perturbed by a mass-preserving multiplicative space-time white noise;
see~\eqref{eq:50} below.  Equations of fluctuating hydrodynamics are widely used in physics and other sciences (e.g., in the
description of active matter~\cite{Thompson2011a,Cates2015a}, thermal advection~\cite{Lutsko2012a}, neural
networks~\cite{Rotskoff2018a}, and agent based models~\cite{Donev2014a}), and are currently being investigated
numerically~\cite{Helfmann2019a}.  Still, the mathematical analysis of these equations is in its infancy.  A truly remarkable
recent result~\cite{Konarovskyi2019a} shows that a solution for the original Dean--Kawasaki model (as derived in~\cite{Dean1996a}
and given in~\eqref{eq:50} below) only exists when the initial datum is a superposition of a finite-number $N$ of Dirac delta
functions and the diffusion coefficient is $\frac 12 N$; if such an initial datum is ever so slightly mollified, then no solution
exists.  Given the numerous applications of equations of fluctuating hydrodynamics, this apparent mathematical instability is
particularly puzzling.

In light of this, several \emph{regularised} Dean--Kawasaki models (featuring smooth noise coefficient and coloured driving noise)
have been proposed and studied~\cite{Fehrman2019a,Donev2014a,Gess2016a,Cornalba2019a,Cornalba2020a}.  In recent
work~\cite{Cornalba2019a,Cornalba2020a}, the authors have derived and analysed stochastic PDE models for the empirical density of
$N$-particles following second-order Langevin dynamics and interacting weakly. The models are derived from particles as entities
of finite size rather than Dirac delta functions and this regularisation is crucial for the mathematical theory. We refer to this
PDE as the Regularised Inertial Dean--Kawasaki (RIDK) model.  In particular, we have established that RIDK has a well-defined mild
solution in one-dimension with probability converging to one in the limit as $N\to\infty$ and the particle width $\epsilon\to 0$,
subject to particles being wide enough (as given by the scaling condition $N\,\epsilon^\theta= 1$ for a given $\theta$). In this
paper, we establish well-posedness for RIDK in any finite spatial dimension and significantly improve the scaling condition (relax
conditions on $\theta$) in the one-dimensional case.  To the best of our knowledge, this is the first proof of well-posedness for
RIDK or any Dean--Kawasaki model in several space dimensions.

\subsection{Setting and main result}
\label{sec:Setting-main-result}
We consider $N$-weakly interacting particles on the $d$-dimensional torus $\mathbb{T}^d\coloneqq [0,2\pi)^d$. The particles are
  identified by position and momentum $(\f{q}_i,\f{p}_i)_{i=1}^{N}\in\mathbb{T}^d\times \mathbb{R}^d$, and satisfy the stochastic
  differential equation
\begin{align}
  \label{eq:19}
  &    \displaystyle \dot{\f{q}}_i= \f{p}_i,  \qquad
  \dot{\f{p}}_i= -\gamma\, \f{p}_i-N^{-1}\sum_{j=1}^{N}{\nabla U(\f{q}_i-\f{q}_j)}+\sigma\,\dot{\f{b}}_i,\qquad i=1,\dots,N,
\end{align}
where $\gamma,\sigma$ are positive constants, $U\colon\mathbb{T}^d\rightarrow \mathbb{R}$ is a smooth pairwise interaction
potential, and $\{\f{b}_i\}_{i=1}^{N}$ is a family of independent standard $d$-dimensional Brownian motions.  We work under the
key modelling assumption that the particles have a finite size.  Specifically, we describe their spatial occupancy by means of a
kernel $w_{\ep}\colon\mathbb{T}^d\rightarrow [0,\infty)$ indexed by $\ep>0$, which may be thought of as the particle width.  We
  propose RIDK as a model for the particle density and momentum density
   \[
     (\rho_{\ep}(\f{x},t),\f{j}_{\ep}(\f{x},t))
     \coloneqq  \pp{N^{-1}\sum_{i=1}^{N}{w_{\ep}(\f{x}-\f{q}_i(t))},
  N^{-1}\sum_{i=1}^{N}{\f{p}_i(t)\,w_{\ep}(\f{x}-\f{q}_i(t))}},
  \]
   where $(\f{x},t)\in\mathbb{T}^d\times[0,T]$. 

In particular, RIDK defines an approximate particle and momentum density
$(\tilde{\rho}_{\ep},\tilde{\f{\jmath}}_{\ep})\colon\mathbb{T}^d\times[0,T]\rightarrow \mathbb{R}\times\mathbb{R}^d$ by the
stochastic PDE
\begin{equation}
 \left\{\!
    \begin{array}{l}
      \partial_t \tilde{\rho}_{\ep} = -\nabla \cdot \tildebj_\ep,  \\
      \partial_t \tildebj_\ep= -\gamma \,
      \tildebj_{\ep}-\dfrac{\sigma^2}{2\gamma}\,\nabla\tilde{\rho}_{\ep}-\tilde{\rho}_{\ep}\,(\nabla U\ast\tilde{\rho}_{\ep})
      +\sigma\, N^{-1/2}\left(\sqrt{\tilde{\rho}_{\ep}}\,
      P^{1/2}_{\sqrt{2}\ep}\,\xi_{1},\dots,\sqrt{\tilde{\rho}_{\ep}}\,
      P^{1/2}_{\sqrt{2}\ep}\,\xi_{d}\right)\!,
    \end{array}
    \label{eq:23}
    \right.
\end{equation}
subject to $(\tilde{\rho}_\ep(\cdot,0),\tildebj_\ep(\cdot,0))=(\tilde{\rho}_0,\tildebj_0)$ for initial densities $\tilde{\rho}_0$
and $\tilde{\f{j}}_{0}$, where $\{\xi_{\ell}\}_{\ell=1}^{d}$ are independent space-time white noises, and $P_{\ep}$ is the
convolution operator $P_{\ep}\colon L^2(\mathbb{T}^d)\rightarrow L^2(\mathbb{T}^d)\colon f\mapsto
P_{\ep}\,f(\cdot)=\int_{\mathbb{T}^d}{w_{\ep}(\cdot-\f{y})\,f(\f{y})\,\m \f{y}}$.  The operator $P_{\ep}$ describes the spatial
correlation of the stochastic noise and is intrinsically linked to the spatial occupancy of the particles through the regularising
kernel $w_{\ep}$. This model is of inertial type (meaning that it keeps track of both density and momentum density), and is a
generalisation of the models studied in~\cite{Cornalba2019a,Cornalba2020a} to higher dimensions. For $w_\epsilon$, we choose the
von Mises kernel
\begin{align}
  \label{eq:2001}
  w_{\ep}(\f{x})\coloneqq Z_{\ep}^{-d}\exp\left\{-\frac{\sum_{\ell =
      1}^{d}{\sin^2(x_{\ell}/2)}}{\ep^2/2}\right\}, 
      \qquad
       Z_{\ep}\coloneqq
        \int_{\mathbb{T}} \exp\pp{-\frac{\sin^2(y/2)}{\ep^2/2}}\,\m y.
\end{align}
Any  non-negative function  $w\colon [-\pi,\pi)\to \mathbb{R}$ can be written $w(x)=
\exp(- V(\sin(x/2)))$ for a function $V\colon [-1,1]\to \mathbb{R}$. For $x\approx 0$, $V(\sin(x/2))\approx V(0)+V'(0) \sin(x/2)+\frac 12 V''(0) \sin^2(x/2)$.  Assuming  $V$ is symmetric for extending periodically, we find  that \[w(x)\approx \exp(-V(0)) \exp\pp{-\frac12 V''(0) \sin^2(x/2)}.\]
The values of $\exp(-V(0))=1/Z_\epsilon$ and $V''(0)=4/\epsilon^2$ are chosen so that the moments agree with $N(0,\epsilon^2)$ and there is convergence to the Dirac delta function. The periodic extension to $\mathbb{T}^d$ defines the von Mises kernel $w_\epsilon$.

For regularity purposes which will become clear later, it is convenient to replace the square-root in~\eqref{eq:23} with a smooth
function $h_{\delta}\colon \mathbb{R}\rightarrow \mathbb{R}$ such that $h_{\delta}(z)=\sqrt{|z|}$ for $|z|\geq \delta/2$, for some
small and fixed $\delta>0$. Following this change, the RIDK equation~\eqref{eq:23} is rewritten in the abstract stochastic PDE
notation
\begin{equation}
  \label{eq:25}
  \left\{\quad
    \begin{aligned}
      \m X_{\ep,\delta}(t)&=A\, X_{\ep,\delta}(t)\,\m t +\alpha_{U}(X_{\ep,\delta}(t))\,\m t+B_{N,\delta}(X_{\ep,\delta}(t))\,\m W_{\per,\ep}, \\
      X_{\ep,\delta}(0)&=X_0,
    \end{aligned}
  \right.
\end{equation}
where $X_{\ep,\delta}=(\tilde{\rho}_{\ep,\delta},\tildebj_{\ep,\delta})$, $X_0=(\tilde{\rho}_{0},\tildebj_0)$, $A$ is a linear
operator describing the deterministic drift excluding the interaction-potential term $\alpha_{U}(X_{\ep,\delta})$, $B_{N,\delta}$
is the stochastic integrand associated with the introduction of $h_{\delta}$, and $W_{\per,\ep}$ is a $Q$-Wiener representation of
the noise \( \left(P^{1/2}_{\sqrt{2}\ep}\,\xi_{1},\dots,P^{1/2}_{\sqrt{2}\ep}\,\xi_{d}\right) \).  More details
concerning~\eqref{eq:25}, as well as a sketch of its derivation from the Langevin particle dynamics~\eqref{eq:19}, are given in
Section~\ref{s:3}.
 
Throughout the paper, we work under the general scaling
\begin{align}
  \label{eq:2000}
  N\,\ep^{\theta}=1,\qquad \theta>\theta_0 \coloneqq 2d.
\end{align}

From a modelling point of view,~\eqref{eq:2000} imposes the particle size (comparatively to $N$), where increasing $\theta$
implies increasing particle size. If $\theta$ is close to the limiting case $\theta_0=2d$, the scaling~\eqref{eq:2000} is
approximately only dependent on the volume $v$ of each particle: specifically, since each particle is roughly of size $\ep$ in
each direction, then~\eqref{eq:2000} corresponds to $Nv^2\approx 1$. The purpose of the scaling is regularisation and smoothing of
the densities, and it is necessary that particles overlap in the limit so that conservation of volume is not possible. The
condition~\eqref{eq:2000} provides sufficient regularity for the Sobolev space analysis in Theorem~\ref{thm:10} and is optimal in
that sense as we show in Remark~\ref{rem:120}.

From an analytical perspective,~\eqref{eq:2000} affects the spectral properties of the noise
$W_{\per,\ep}$ in~\eqref{eq:25} through the operator $P_{\ep}$.

We state the main result of this paper.

\begin{theorem}[Well-posedness of RIDK on $\mathbb{T}^d$]
  \label{thm:10}
  Let $\delta>0$, $h_{\delta}\in C^{\lceil d/2 \rceil+2}(\mathbb{R})$, $\nu\in(0,1)$, and $U\in C^1$. Fix $\theta>\theta_0=2d$
  such that $(\theta-\theta_0)/2<\lceil (d+1)/2 \rceil-d/2$. Pick $\eta\in(0,\min\{(\theta-\theta_0)/2, C(d)\})$ for some small
  enough $C(d)\in(0,1/2)$ (see Lemma~\ref{gagl}). Set $s\coloneqq d/2+\eta$.  Let $(\tilde{\rho}_0,\tildebj_0)$ be a deterministic
  initial condition belonging to the fractional Sobolev space $\mathcal{W}^s\coloneqq H^s(\mathbb{T}^d)\times
  [H^s(\mathbb{T}^d)]^d$ such that $\min_{\f{x}\in\mathbb{T}^d}\tilde{\rho}_0(\f{x})> \delta$.
  
  There exists $T=T(\tilde{\rho}_0)$, a large enough $N$, a unique $\mathcal{W}^s$-valued process
  $X_{\ep,\delta}=(\tilde{\rho}_{\ep,\delta},\tildebj_{\ep,\delta})$, and a set $F_{\nu}$ of probability at least $1-\nu$ such
  that $\min_{\f{x}\in\mathbb{T}^d,s\in[0,T]}\tilde{\rho}_{\ep,\delta}(\f{x},s)\geq \delta$ on $F_{\nu}$, and $X_{\ep,\delta}$
  solves~\eqref{eq:25} pathwise on $F_{\nu}$ in the sense of mild solutions~\cite[Chapter 7]{Da-Prato2014a}. As a consequence,
  $X_{\ep,\delta}$ also solves the RIDK equations~\eqref{eq:23} pathwise on $F_{\nu}$ in the sense of mild solutions.
\end{theorem}

The proof exploits a small-noise analysis, by obtaining the solution to~\eqref{eq:25} as a small perturbation of the strictly
positive solution of the noise-free dynamics (i.e., the damped wave equation). When the perturbations are small and the initial
data is everywhere larger than $\delta$, the solution to~\eqref{eq:25} remains outside the regularisation regime
$(-\infty,\delta/2)$ for $h_\delta$ and the regularisation is bypassed, resulting in a well-defined solution of~\eqref{eq:23}. The
$C^0$-norm is used to measure the perturbations and keep track of whether the solution falls into the regularisation region. To do
this, the parameter $s$ is chosen so the mild solutions take values in the Sobolev space $\mathcal{W}^s$, which is embedded
continuously in $C^0\times [C^0]^d$.

With this in mind, the proof of Theorem~\ref{thm:10} (see Section~\ref{ss:20}) is built upon three conceptual blocks, developed in
Section~\ref{s:2}. Firstly, $A$ is proven to generate a $C_0$-semigroup with respect to the $\mathcal{W}^s$-norm (see
Subsection~\ref{ss:1}, Lemma~\ref{lem:80}). Secondly, the stochastic integrand $B_{N,\delta}$ is shown to be locally Lipschitz and
sublinear ($d=1$) or locally Lipschitz and locally bounded ($d>1$) with respect to the Hilbert--Schmidt
$L^0_2(\mathcal{W}^s)$-norm (Subsection~\ref{ss:3}, Lemma~\ref{lem:7}). These two blocks give rigour to the application of the
mild solution theory. Thirdly, sharp bounds for the trace of $W_{\per,\ep}$ with respect to the $\mathcal{W}^s$-norm are provided
via spectral analysis of $P_{\sqrt{2}\ep}$ (see Subsection~\ref{ss:2}, Lemma~\ref{lem:1}). In combination with Lemma~\ref{lem:7},
this guarantees the vanishing-noise regime for~\eqref{eq:25} in the $\mathcal{W}^s$-norm as $N\rightarrow \infty$.

Theorem~\ref{thm:10} carries two significant contributions. Firstly, it provides a well-posedness theory for the multi-dimensional
RIDK model; to the best of our knowledge, this is the first paper to give an existence and uniqueness theory for such a model.
Secondly, it improves an analogous one-dimensional result~\cite{Cornalba2019a,Cornalba2020a} by significantly relaxing the scaling
threshold in~\eqref{eq:2000} from $\theta_0=7$ to $\theta_0=2$. The more restrictive threshold for $\theta$ resulted from a
suboptimal analysis with respect to the $\mathcal{W}^1$ norm. The $\theta_0=7$ scaling is inconveniently restrictive, as it only
allows for rather large particles (comparatively to $N$). Specifically, $\theta$ is significantly away from the null value, which
\emph{formally} corresponds to representing particles by Dirac delta functions.

The main technical novelties that we introduce in the proof of Theorem~\ref{thm:10} are the following. First, we deploy improved
estimates for the spectral properties of $P_{\sqrt{2}\ep}$, which rely on refined bounds for modified Bessel functions of the
first kind. Secondly, we set the analysis in the `least restrictive' Sobolev space $\mathcal{W}^s$ that embeds continuously in the
space of continuous functions, and this corresponds to considering $s=d/2+\eta$ for arbitrarily small positive $\eta$. Thirdly, we
extend the analysis to higher dimensions by relying on the separability of the kernel $w_{\ep}$ in its $d$ variables, the boundedness and periodicity of the spatial domain, and the
considerations from the one-dimensional case. The boundedness of the spatial domain is crucially used also in the derivation of technical tools related to fractional Sobolev spaces and Fa\`a di Bruno's formula, which are deferred to Appendix~\ref{app:b}. Relevant elementary algebraic tools are summarised in
Appendix~\ref{app:a}. Additionally, the proof of Theorem~\ref{thm:10} in Section~\ref{ss:20} is finalised with a localisation
procedure argument. Crucially, the same techniques adopted to deal with the superlinear interaction $\alpha_{U}$ (analogous to
those developed in~\cite[Section 4]{Cornalba2020a}) also allow to deal with the locally bounded noise in $d>1$.

\begin{rem}
  \label{rem:120}
  The justification of the scaling assumptions of Theorem~\ref{thm:10} is found in Lemma~\ref{lem:1}. There, each index $s$ is
  associated to a relevant value $\theta_{c}(s)\coloneqq 2s+d$, and the trace of $W_{\per,\ep}$ with respect to the
  $\mathcal{W}^s$-norm is bounded by $\ep^{-\theta_c(s)}$. In combination with Lemma~\ref{lem:7}, this implies that the
  $\mathcal{W}^{s}$-norm of the stochastic noise of~\eqref{eq:25} vanishes as $N\rightarrow \infty$ for any $\theta >
  \theta_c(s)$. As our well-posedness theory relies on the embedding $\mathcal{W}^s\subset C^0\times [C^0]^d$, we require the
  equality $s=d/2+\eta=(\theta_c(s)-d)/2$ to hold, giving $\theta>2d+2\eta$. As $\eta$ may be chosen arbitrarily small, we obtain
  the threshold $\theta_0=2d$.

  Furthermore, for each $s$, the value $\theta_c(s)$ is optimal, in the sense that $\theta_c(s)$ is also the minimum value for
  which $\mean{\|\rho_{\ep}(\cdot,t)\|^2_{H^s}}$ (where, we recall, $\rho_{\ep}$ denotes the true particle density) is uniformly
  bounded in $N$ and $\ep$, at least in the case of independent particles given by $U\equiv 0$. Namely, it is easy to proceed as
  in~\cite{Cornalba2019a} and argue that, under reasonable assumptions on the law of the particle dynamics,
  \begin{equation*}
    0<C_1<\lim_{N\to\infty,\,\ep\to 0}\left\{N\,\ep^{2s+d}\,\mean{\|\rho_{\ep}(\cdot,t)\|^2_{H^s}}\right\}<C_2,
  \end{equation*} 
  for some constants $C_1$, $C_2$ independent of $N$ and $\ep$. Crucially, we obtain scaling agreement for fluctuations on
  microscopic and mesoscopic scale; here \emph{microscopic} means particle-level dynamics, see~\eqref{eq:19} above, while
  \emph{mesoscopic} means the Dean--Kawasaki dynamics~\eqref{eq:23}. As a result, the value $\theta_0$ is also optimal, as
  $\lim_{s\rightarrow d/2}\theta_c(s)=\theta_0$.  We stress that this notion of scaling optimality is only understood with respect
  to the evaluation of Sobolev norms for the densities $\rho_\ep$ and $\tilde{\rho}_{\ep}$.
  
\end{rem}

\begin{rem}
  The RIDK model~\eqref{eq:23} may be regarded as the regularised inertial analogue of the original over-damped Dean--Kawasaki
  model~\cite{Dean1996a,Kawasaki1998a}
  \begin{align}\label{eq:50}
    \partial_t \rho_N=\frac{N}{2}\,\Delta \rho_N+\nabla\cdot\left(\sqrt{\rho_N}\,\xi\right)\quad
    \mbox{in }\mathbb{R}^d\times (0,\infty),
  \end{align}
  where $\xi$ is a space-time white noise. As stated before,~\eqref{eq:50} admits nothing but an atomic solution, and only in the
  integer regime $N\in\mathbb{N}$. In this case, the solution is $\rho_N(\f{x},t)=N^{-1}\sum_{i=1}^{N}{\delta(\f{x}-\f{B}_i(t))}$,
  where $(\f{B}_i)_{i=1}^{N}$ are independent Brownian walkers.

We consider the RIDK model~\eqref{eq:23} instead of~\eqref{eq:50} for a number of reasons. Firstly, while both models describe
mesoscopic fluctuations in particle systems of physical relevance (such as those treated, e.g., in the description of active
matter~\cite{Thompson2011a,Cates2015a} and thermal advection~\cite{Lutsko2012a}), the RIDK model does so while also capturing core
inertial effects (i.e., Newton's law of motion).  Secondly, it bypasses any problematic interpretations arising from taking a
formal divergence of the stochastic noise. Finally, it allows to work with smooth rather than atomic solutions. 
\end{rem}

\subsection{Comparison with classical over-damped Dean--Kawasaki model and open problems}

A rigorous connection between~\eqref{eq:23} and~\eqref{eq:50} is, to the present day, still lacking.  Firstly, it is not at all
clear if~\eqref{eq:50} can be recovered via an over-damped limit (i.e., by taking $\gamma\rightarrow \infty$) in~\eqref{eq:23}: to
the best of our knowledge, there are no known mathematical results in this context, and, additionally,~\eqref{eq:23} admits smooth
solutions while~\eqref{eq:50} only admits atomic solutions.  Secondly, we have no clear indication as to how to recover the square
root singularity in~\eqref{eq:23} (i.e., how to perform the limit $\delta\rightarrow 0$), even in the context of strictly positive
solution to the noise-free dynamics of~\eqref{eq:23} considered in the paper.  The closest result on the subject is given
in~\cite{Fehrman2019a}. The macroscopic limit $N\rightarrow \infty$ (which, in the case of~\eqref{eq:23}, also forces the removal
of the regularisation $\ep\rightarrow 0$ due to the scaling~\eqref{eq:2000}) is, on the contrary, better understood. The solution
to~\eqref{eq:23} converges to the solution of a noiseless wave equation as a consequence of Theorem~\ref{thm:10}; on the other
hand, the solution to~\eqref{eq:50} converges to the solution of a deterministic parabolic equation, at least for reasonable
initial configurations (this follows from the definition of $\rho_N$ and the law of large numbers).

In the case $\theta\leq 2d$ (currently out of the scope of our well-posedness theory), or indeed for any other scaling of $N$ and
$\ep$, the over-damped limit in~\eqref{eq:23} is just as open a question as for the case $\theta>2d$, while the macroscopic limit
$N\rightarrow \infty$ is unknown.

We now briefly turn to possible future improvements of our RIDK model. The derivation of the RIDK model~\eqref{eq:23} heavily
relies on boundedness and periodicity of the spatial domain $\mathbb{T}^d$. In this case, the spectrum of the convolution operator
$P_\ep$ is known: as explained in Section \ref{s:3}, this is a consequence of the one-dimensional analysis treated
in~\cite{Cornalba2020a}, of suitable multiplication rules for the kernel $w_\ep$, and of its separability in the $d$ variables on
$\mathbb{T}^d$.

The analysis of the RIDK model takes the spectral properties of $P_\ep$ merely as starting points, and, therefore, is a relatively
independent and self-contained argument. It is reasonable to expect that it could be extended to general bounded domains (and, if
applicable, to different boundary conditions) by adapting the analysis of the spectrum of $P_\ep$. Improvements on the scaling
requirement~\eqref{eq:2000} will likely come from using less restrictive notions of solutions.  Questions such as the long time
behaviour of solutions and the invariant measures for~\eqref{eq:23} are hard to answer, and will likely require a radically
different approach.

\subsection{Basic notation}
\label{ss:0}

We work with periodic functions on the $d$-dimensional torus $\mathbb{T}^d=[0,2\pi)^d$ for $d\in\mathbb{N}$. We never specify the
dependence of any function space on $d$, as this is always clear from the context. Bold face characters always denote vectors. For
$m\in\mathbb{N}_0$ and $p\in[1,\infty]$, we denote by $W^{m,p}$ the standard Sobolev space of periodic functions on $\mathbb{T}^d$
with derivatives up to order $m$ belonging to $L^p$. For $0<s\notin \mathbb{N}$ and $p\in[1,\infty)$, we define the fractional
spaces $W^{s,p}$ via the norm
\begin{equation}
  \label{eq:100}
  \|u\|_{W^{s,p}}\coloneqq \|u\|_{W^{\lfloor s\rfloor,p}}+\max_{|\f{z}|=\lfloor s\rfloor}\int_{\mathbb{T}^d}{\int_{\mathbb{T}^d}{\frac{|\partial^{\f{z}}
        u(\f{x})-\partial^{\f{z}}u(\f{y})|^p}
      {|\f{x}-\f{y}|^{d+(s-\lfloor s\rfloor)p}}\,\m \f{x}}\,\m \f{y}},
\end{equation} 
where $\lfloor s\rfloor\coloneqq \max\{n\in\mathbb{N}_0\colon n\leq s\}$. We also set $\lceil y\rceil\coloneqq
\min\{n\in\mathbb{N}:y\leq n \}$. We consider the fractional Hilbert spaces $H^s$ and $\f{H}^s\coloneqq \left[H^s\right]^d$
identified by the Fourier-type inner products
\begin{align}
  \label{eq:101}
  \langle u,v\rangle_{H^s} & \coloneqq \sum_{\f{j}\in\mathbb{Z}^d}{\hat{u}_{\f{j}}\,\overline{\hat{v}_{\f{j}}}\,(1+|\f{j}|^2)^s},\qquad
  \hat{u}_{\f{j}}\coloneqq (2\pi)^{-d}\int_{\mathbb{T}^d}{e^{-\mathrm{i}\,\f{j}\cdot \f{x}}\,u(\f{x})\,\m \f{x}},\qquad u,v\in H^s,\nonumber\\
  \langle \f{u},\f{v}\rangle_{\f{H}^s} & \coloneqq  \sum_{\ell=1}^{d}{\langle u_{\ell},v_{\ell}\rangle_{H^s}},\qquad \f{u},\f{v}\in \f{H}^s,
\end{align}
and we define the norm on $\mathcal{W}^s\coloneqq H^s\times \f{H}^s$ as $\|(u,\f{v})\|_{\mathcal{W}^s}\coloneqq
(\|u\|^2_{H^s}+\|\f{v}\|^2_{\f{H}^s})^{1/2}$, for $(u,\f{v})\in\mathcal{W}^s$. The norms $\|\cdot\|_{H^s}$ and
$\|\cdot\|_{W^{s,2}}$ are equivalent; see~\cite[Proposition 1.3]{Benyi2013a}. We define the space $ \f{V}^{s+1}\coloneqq
\left\{\f{v}\in\f{H}^s\colon \nabla\cdot\f{v}\in H^s\right\}\supset \f{H}^{s+1}$, and recall the integration-by-parts formula
\begin{align}
  \label{eq:103}
  \langle -\nabla\cdot \f{v},u\rangle_{H^s}=\langle \f{v},\nabla u\rangle_{\f{H}^s},\qquad\forall u\in H^{s+1},\forall\f{v}\in \f{V}^{s+1}.
\end{align}
In dimension $d=1$, we trivially have $V^{s+1}\equiv H^{s+1}$. We denote by $L(\mathcal{W}^s)$ (respectively,
$L^0_2(\mathcal{W}^s)$) the set of continuous linear functionals mapping $\mathcal{W}^s$ into itself (respectively, the set of
Hilbert--Schmidt operators from $P^{1/2}_{\sqrt{2}\ep}\mathcal{W}^s\subset\mathcal{W}^s$ into $\mathcal{W}^s$), with the
convolution operator $P_{\sqrt{2}\ep}$ as defined after~\eqref{eq:23}.

For each $\alpha\in\mathbb{N}$, we define
\begin{align*}
  \Pi_{\alpha} & \coloneqq \left\{\mbox{set of partitions of }\{1,\dots,\alpha\}\right\},\\
  B(\pi) & \coloneqq  \left\{\mbox{set of blocks forming partition }\pi\right\},\qquad \pi\in \Pi_\alpha,\\
  |\pi| & \coloneqq  \#B(\pi)=\text{number of blocks forming partition $\pi$},
\end{align*}
where $\#$ denotes the number of elements in a set.  Furthermore, for every partition $\pi\in\Pi_{\alpha}$, we set
\begin{align*}
  \beta_j(\pi) & \coloneqq  \#\{b\in B(\pi)\colon |b|=j\},\qquad j\in\{1,\dots,\alpha\},\\
  J(\pi) & \coloneqq  \{j\in\{1,\dots,\alpha \}\colon \beta_j(\pi)>0\}.
\end{align*}
As an immediate consequence of the definitions, we have $\sum_{j\in J(\pi)}{j\,\beta_j(\pi)}=\alpha$.

We use $C$ as a generic constant whose value may change from line to line (with dependence on relevant parameters highlighted
whenever necessary, for example $C(s)$). In addition, we denote the embedding constant of $H^s\subset C^0$ by $\K$. Finally, we
use the subscript notation to link specific constants with the lemmas where they are defined; for example, $K_{\ref{prod}}$ is the
constant introduced in Lemma~\ref{prod}.

\section{Derivation of RIDK}
\label{s:3}
We now derive the RIDK model~\eqref{eq:25} by following the methodology outlined in~\cite{Cornalba2020a}. Consider the
second-order Langevin system~\eqref{eq:19}, as well as the quantities
$(\rho_{\ep}(\f{x},t),\f{j}_{\ep}(\f{x},t))=(N^{-1}\sum_{i=1}^{N}{w_{\ep}(\f{x}-\f{q}_i(t))},
N^{-1}\sum_{i=1}^{N}{\f{p}_i(t)w_{\ep}(\f{x}-\f{q}_i(t))})$ defined via the kernel~\eqref{eq:2001}.  Simple It\^o computations
imply that $\rho_{\ep}$ and $\f{j}_{\ep}$ satisfy the system
\begin{equation}
  \label{eq:20}
  \left\{
    \begin{array}{l}
           \displaystyle \partial_t \rho_{\ep}(\f{x},t)  = -\nabla \cdot \f{j}_\ep(\f{x},t),  \\
      \displaystyle\partial_t \f{j}_\ep(\f{x},t)  = -\gamma\, 
      \f{j}_{\ep}(\f{x},t) -\f{j}_{2,\ep}(\f{x},t) + \f{I}_{U}(\f{x},t)
      +\f{\dot{Z}}_N(\f{x},t) ,
    \end{array}
  \right.
\end{equation}
where the $\ell$th component of terms on the right are defined by
\begin{align*}
  [\f{j}_{2,\ep}(\f{x},t) ]_\ell & \coloneqq  N^{-1}\sum_{i=1}^{N}{p_{\ell,i}^2(t)\,\partial_{x_\ell}w_{\ep}(\f{x}-\f{q}_i(t))}
  +N^{-1}\sum_{i=1}^{N}{\sum_{k\neq \ell}{p_{\ell,i}(t)\,p_{k,i}(t)\,\partial_{x_\ell}w_{\ep}(\f{x}-\f{q}_i(t))}},\\
  [\f{I}_{U}(\f{x},t)]_{\ell} & \coloneqq  -N^{-1}\sum_{i=1}^{N}{N^{-1}\sum_{j=1}^{N}{\partial_{x_\ell}U(\f{q}_i(t)
      -\f{q}_j(t))}\,w_{\ep}(\f{x}-\f{q}_i(t))},\nonumber\\
  \left[\f{\dot{Z}}_N(\f{x},t)\right]_{\ell} & \coloneqq  \sigma N^{-1}\sum_{i=1}^{N}{w_{\ep}(\f{x}-\f{q}_i(t))\,\dot b_{\ell,i}}.\nonumber
\end{align*}
The terms $\f{j}_{2,\ep},\f{I}_{U}$, and $\f{\dot{Z}}_N$ are not closed in the leading quantities $(\rho_{\ep},\f{j}_{\ep})$, and  approximations are used to close the system of equations. We now sketch how the approximations in~\cite{Cornalba2019a,Cornalba2020a} extend to the
multi-dimensional case.

The term $\f{j}_{2,\ep}$ is dealt with under a local-equilibrium assumption~\cite[Corollary 3.2]{Duong2018b}. In this situation,
the probability density function of $(\f{q}_i(t),\f{p}_i(t))$ is approximately separable in the position variable $\f{q}_i(t)$ and
momentum variable $\f{p}_i(t)$ due to the structure of the Gibbs invariant measure. In addition, the momentum variable is
distributed according to a Gaussian of mean zero and diagonal covariance matrix $(\sigma^{2}/2\gamma)\,I_d$. Furthermore, under
the additional assumption $\sigma^2\ll 2\gamma$, the approximation $\sigma^2/(2\gamma)=\mathbb{E}[p^2_{\ell,i}(t)]\approx
p^2_{\ell,i}(t)$ is legitimate. All these considerations imply that
$\mean{\f{j}_{2,\ep}}\approx\sigma^2/(2\gamma)\,\mean{\nabla\rho_{\ep}}$ and this motivates the replacement $\f{j}_{2,\ep}\approx
(\sigma^2/2\gamma)\,\nabla\rho_{\ep}$.

The interaction term $\f{I}_{U}$ may be approximated as $\f{I}_{U}\approx -\rho_{\ep}\,(\nabla U\ast\rho_{\ep})$, following the
lines of~\cite[Proposition 3.5]{Cornalba2020a}.

Finally, one may substitute the noise $\f{\dot{Z}}_N(\f{x},t) $ with $\f{\dot{Y}}_N(\f{x},t) $, where
\begin{equation}
  \label{eq:1000}
  \left[\f{\dot{Y}}_N(\f{x},t) \right]_{\ell} \coloneqq  \sigma\, N^{-1/2}\,\sqrt{\rho_{\ep}(\f{x},t)}\,
  P^{1/2}_{\sqrt{2}\ep}\,\xi_{\ell}(\f{x},t) ,
\end{equation}
where $P_{\ep}$ is the convolution operator $P_{\ep}\colon L^2(\mathbb{T}^d)\rightarrow L^2(\mathbb{T}^d)\colon f\mapsto
P_{\ep}f(\cdot)=\int_{\mathbb{T}^d}{w_{\ep}(\cdot-\f{y})\,f(\f{y})\,\m \f{y}}$ introduced above, and where
$\{\xi_{\ell}\}_{\ell=1}^{d}$ are independent space-time white noises. The substitution of $\f{\dot{Z}}_N(\f{x},t) $ with
$\f{\dot{Y}}_N(\f{x},t)$ relies on the two noises being approximately equivalent in distribution. This is a consequence of the following approximate multiplication rule for von Mises kernels
\begin{align}\label{eq:5000}
w_{\ep}(\f{x}_1-\f{q}_i(t))w_{\ep}(\f{x}_2-\f{q}_i(t))\approx w_{\sqrt{2}\ep}(\f{x}_1-\f{x}_2)w_{\ep/\sqrt{2}}\left(\frac{\f{x}_1+\f{x}_2}{2}-\f{q}_i(t)\right),
\end{align}
where $\f{x}_1,\f{x}_2\in\mathbb{T}^d$, which can be deduced from its one-dimensional analogue~\cite{Cornalba2020a} thanks to the crucial fact that the  kernel $w_{\ep}$ is separable in its $d$ variables on $\mathbb{T}^d$.

In addition, the
stochastic independence of the $d$ components of each member of the family $\{\f{b}_i\}_{i=1}^{N}$ is reflected in the stochastic
independence of the $\{\xi_{\ell}\}_{\ell=1}^d$.  Taking all into account, we obtain our multi-dimensional RIDK
system~\eqref{eq:23}.

The noise $\f{\dot{Y}}_N(\f{x},t) $ can be explicitly expanded using the spectral properties of the operator $P_{\ep}$, which, due
to the separability of the kernel $w_{\ep}$, are readily available from the one-dimensional case~\cite[Section
  4.2]{Cornalba2019a}. More specifically, with $\{e_j\}_{j\in\mathbb{Z}}$ being the trigonometric system
\begin{align*}
 e_{j}(x) \coloneqq 
  \begin{cases}
    \sqrt{\frac{1}{\pi}}\cos(jx), & \mbox{if } j>0, \vspace{0.3 pc}\\
    \sqrt{\frac{1}{\pi}}\sin(jx), & \mbox{if } j<0, \vspace{0.3 pc}\\
    \sqrt{\frac{1}{2\pi}}, & \mbox{if } j=0,
  \end{cases}                   
\end{align*}
it is not difficult to see that the family $\{f_{\f{j},s}\}_{\f{j}\in\mathbb{Z}^d}$ defined as
\begin{align}
  \label{eq:411}
  f_{\f{j},s}(\f{x}) \coloneqq  
    C(d)\left\{\prod_{\ell=1}^{d}{e_{j_{\ell}}(x_{\ell})}\right\}\left(1+|\f{j}|^2\right)^{-s/2}, \quad \f{j}\in\mathbb{Z}^d,
\end{align}
is, for some choice of normalisation constant $C(d)$, an $H^s$-orthonormal basis of eigenfunctions for $P_{\sqrt{2}\ep}$ for any
$\ep>0$. Furthermore, the eigenvalue of $P_{\sqrt{2}\ep}$ corresponding to the eigenfunction $f_{\f{j},s}$ is
\begin{equation}
  \label{ev-prod}
  \lambda_{\f{j},\ep}  =\prod_{\ell=1}^{d}{\lambda_{j_{\ell},\ep}},
\end{equation}
where the eigenvalues from the one-dimensional case are given by
\begin{align}
  \label{eq:2005}
  \lambda_{j,\ep} 
  & =
    \begin{cases}
      \displaystyle Z^{-1}_{\sqrt{2}\ep}\int_{\mathbb{T}}{e^{-\frac{\sin^2(x/2)}{\ep^2}}\cos(jx)
      \m x}=
   I_{j}\left(\{2\ep^2\}^{-1}\right) / I_0(\Bp{2\epsilon^2}^{-1} ), 
      & \mbox{if } j\neq 0, \vspace{0.3 pc}\\
      1, & \mbox{if } j=0,
    \end{cases}
\end{align}
with $I_j$ denoting the modified Bessel function of first kind and order $j$~\cite[Eq.~(9.6.19)]{Abramowitz1964a}. As a result,
the stochastic process
\begin{align}
  \label{eq:unreg-noise}
  W_{\per,\ep} & \coloneqq \sum_{\f{j}\in\mathbb{Z}^d}{\sqrt{\alpha_{\f{j},s,\ep}}\,(0,f_{\f{j},s},0,\dots,0)}\,\beta_{1,\f{j}}+\cdots\nonumber\\
  & \quad + \sum_{\f{j}\in\mathbb{Z}^d}
  {\sqrt{\alpha_{\f{j},s,\ep}}\,(0,\dots,0,f_{\f{j},s})}\,\beta_{d,\f{j}},\qquad \alpha_{\f{j},s,\ep}
  \coloneqq (1+|\f{j}|^2)^{s}\lambda_{\f{j},\ep},
\end{align}
with \iid families $\{\beta_{\ell,\f{j}}\}_{\ell=1}^{d}$ of independent Brownian motions, is a $\mathcal{W}^s$-valued $Q$-Wiener
process representation of the $\mathbb{R}\times\mathbb{R}^d$-valued stochastic noise $(0,\f{\dot{Y}}_N(\f{x},t))$. It follows
that, upon swapping $\f{\dot{Z}}_N(\f{x},t) $ with $\f{\dot{Y}}_N(\f{x},t) $, we can write~\eqref{eq:20} in the abstract
stochastic PDE form
\begin{equation}
  \label{eq:24}
  \left\{\quad
    \begin{aligned}      \m X_{\ep}(t)&=A\,X_{\ep}(t)\,\m t+\alpha_{U}(X_{\ep,\delta}(t))\,\m t+B_{N}(X_{\ep}(t))\,\m W_{\per,\ep}, \\
      X_{\ep}(0)&=X_0,
    \end{aligned}
  \right.
\end{equation}
where $X_{\ep}=(\tilde{\rho}_{\ep},\tildebj_{\ep})$, $A$ is the wave-type differential operator given by
\[
  A\,X
  \coloneqq \pp{-\nabla\cdot\f{j},
  \,-\gamma \,{\f{j}}
  -(\sigma^2/2\gamma)\,\nabla{\rho} },\qquad 
  X=(\rho,\f{j}),
\]
the interaction potential is $\alpha_{U}(X_{\ep,\delta})\coloneqq -\tilde{\rho}_{\ep}\,(\nabla U\ast\tilde{\rho}_{\ep})$, and the
stochastic integrand $B_N$ is given by
\begin{align*}
  B_{N}({\rho},\f{j})(a,\f{b})
  \coloneqq \sigma\, N^{-1/2}\left(0,\, \sqrt{{\rho}}\,b_1,\dots,
  \sqrt{\rho}\,b_d\right).
\end{align*}
For some $h_{\delta}\in C^{\lceil d/2 \rceil+2}(\mathbb{R})$ regularising the square function in $[0,\delta]$, we substitute $B_N$
with the smoothed stochastic integrand
\begin{align}
  \label{eq:reg-noise}
  B_{N,\delta}(({\rho},\f{j}))(a,\f{b})\coloneqq \sigma \,N^{-1/2}\left(0,\, h_\delta({\rho})\,b_1,
  \dots,h_\delta({\rho})\,b_d\right)
\end{align}
in~\eqref{eq:24}, and we finally obtain the following equation in
$X_{\ep,\delta}=(\tilde{\rho}_{\ep,\delta},\tildebj_{\ep,\delta})$
\begin{equation*}
  \left\{\quad
    \begin{aligned}
      \m X_{\ep,\delta}(t)&=A\,X_{\ep,\delta}(t)\,\m t +\alpha_{U}(X_{\ep,\delta}(t))\,\m t+B_{N,\delta}(X_{\ep,\delta}(t))\,\m W_{\per,\ep}, \\
      X_{\ep,\delta}(0)&=X_0,
    \end{aligned}
  \right.
\end{equation*}
which is exactly~\eqref{eq:25}.

\section{Main technical results for the proof of Theorem~\ref{thm:10}}\label{s:2}
We develop the three main technical tools upon which we base the main proof in Section~\ref{ss:20}. We investigate the cases $d=1$
and $d>1$ separately.

\subsection{Semigroup analysis of operator \texorpdfstring{$A$ in $\mathcal{W}^s$}{}}
\label{ss:1}
\begin{lemma}
  \label{lem:80}
  Let $\mathcal{D}(A)\coloneqq H^{s+1}\times\f{V}^{s+1}$. The operator $A\colon \mathcal{D}(A)\subset \mathcal{W}^s\rightarrow
  \mathcal{W}^s$ defines a $C_0$-semigroup of contractions.
\end{lemma}
\begin{proof}[Proof of Lemma~\ref{lem:80} in dimension $d=1$]
The proof is identical to the one provided in~\cite[Lemma 4.2]{Cornalba2019a}, simply with all relevant spaces $H^{\alpha}$ being
replaced by $H^{\alpha-1+s}$. We assume $\sigma^2/(2\gamma)\coloneqq 1$ for simplicity, even though the proof is analogous for the
general case $\sigma^2/(2\gamma)>0$.

We verify the assumptions of the Hille--Yosida Theorem, as stated in~\cite[Theorem 3.1]{Pazy1983a}.

\emph{Step 1: $A$ is a closed operator, and $\mathcal{D}(A)$ is dense in $\mathcal{W}^s$.} This is easily checked.

\emph{Step 2: The resolvent set of $A$ contains the positive half line.} Let $\lambda>0$. We show that the operator
$A_{\lambda}\coloneqq A-\lambda I$ is injective. Assume that $A_{\lambda}(\rho,j)=(0,0)$. We take the $H^s$-inner product of the
first component of $A_{\lambda}(\rho,j)$ with $\rho$ and of the second component of $A_{\lambda}(\rho,j)$ with $j$, and we obtain
\begin{equation*}
 0= \langle-j'-\lambda\,\rho,\rho\rangle_{H^s}+ \langle-(\lambda+\gamma)\,j-\rho',j\rangle_{H^s}=-\lambda\,\|\rho\|^2_{H^s}
  -(\lambda+\gamma)\,\|j\|_{H^s}^2,
\end{equation*}
where we have used~\eqref{eq:103}.  Since $\lambda,\gamma>0$, we deduce that $(\rho,j)=(0,0)$.  We now show that
$A^{-1}_{\lambda}$ is a bounded operator. Consider $A^{-1}_{\lambda}(a,b)=(\rho,j)$. This implies
\begin{align}
  \lambda \rho & =-a-j',\label{eq:400}\\
  (\lambda+\gamma)\,j & = -b-\rho'.\label{eq:401}
\end{align}
Taking the $H^s$-inner product of~\eqref{eq:400} (respectively, of~\eqref{eq:401}) with $\rho$ (respectively, with $j$), we get
\begin{align}
  \label{eq:310}
  \lambda\,\|(\rho,j)\|^2_{\mathcal{W}^s}\leq \lambda\,\|\rho\|^2_{H^{s}}+(\lambda+\gamma)\,\|j\|^2_{H^{s}}
  =\langle -a,\rho\, \rangle_{H^s} + \langle -b,j \rangle_{H^s}.
\end{align}
We use the Cauchy--Schwartz and Young inequalities to deduce $\|A^{-1}_{\lambda}\|_{\mathcal{L}(\mathcal{W}^s,\mathcal{W}^s)}\leq
\lambda^{-1}$, hence the boundedness of $A^{-1}_{\lambda}$. We now show that $\mbox{Dom}(A^{-1}_{\lambda})$ is dense in
$\mathcal{W}^s$. Let us fix $(a,b)\in H^{s}\times H^{s+1}$. The system of equations $A_{\lambda}(\rho,j)=(a,b)$ reads
\begin{align*}
  -j'-\lambda\,\rho = a, \qquad -(\lambda+\gamma)\,j-\rho' = b,
\end{align*}
which promptly gives 
\begin{align}
  \label{eq:315}
  \frac{\rho''}{\lambda+\gamma}-\lambda\,\rho=a-\frac{b'}{\lambda+\gamma}\in H^{s}.
\end{align}
A Fourier series expansion argument provides existence of a unique solution $\rho\in H^{s+2}$ for~\eqref{eq:315}. From
$-(\lambda+\gamma)\,j=\rho'+b$, we immediately deduce that $j\in H^{s+1}$. We have shown that, for every $(a,b)$ in the dense
subset $H^{s}\times H^{s+1}\subset\mathcal{W}^s$, the operator $A^{-1}_{\lambda}$ is well-defined.

\emph{Step 3: Inequality~\cite[(3.1)]{Pazy1983a} is satisfied:} This is precisely
$\|A^{-1}_{\lambda}\|_{\mathcal{L}(\mathcal{W}^s,\mathcal{W}^s)}\leq \lambda^{-1}$, which we already proved.
\end{proof}

\begin{proof}[Proof of Lemma~\ref{lem:80} in dimension $d>1$]
Steps 2 and 3 are readily adapted, as the Fourier analysis is unchanged. We only need need to justify the validity of Step 1. As
for the density of $\mathcal{D}(A)$ in $\mathcal{W}^s$, this is implied by the density of $H^{s+1}$ into $H^{s}$ and $\f{H}^{s+1}$
into $\f{H}^s$, as well as by the inclusion $\f{H}^{s+1}\subset\f{V}^{s+1}$. As for the closedness of the operator $A$, this
follows from the consistency of the first component of $A$ and of the definition of $\f{V}^{s+1}$. More specifically, consider a
sequence $\mathcal{D}(A)\ni(\rho_n,\f{j}_n)\rightarrow (\rho,\f{j})$ in $\mathcal{W}^s$, such that $A\,(\rho_n,\f{j}_n)\rightarrow
(x,\f{y})$ in $\mathcal{W}^s$. This immediately implies that $\nabla\cdot \f{j}_n$ converges in $H^{s-1}$ to both $-x$ and
$\nabla\cdot\f{j}$, forcing them to agree. In particular, $\f{j}\in \f{V}^{s+1}$. Similarly, $\nabla\rho_n$ converges in
$\f{H}^{s-1}$ to both $\nabla\rho$ and $-\gamma\,\f{j}-\f{y}$, forcing them to agree. In particular, $\rho\in H^{s+1}$. Therefore,
$(\rho,\f{j})\in \mathcal{D}(A)$ and $A\,(\rho,\f{j})=(x,\f{y})$.
\end{proof}

\subsection{Improved bounds on trace \texorpdfstring{of $W_{\per,\ep}$ in $\mathcal{W}^s$-norm}{}}%
\label{ss:2}

\begin{lemma}
  \label{lem:1}
  Let $\{\lambda_{\ep,\f{j}}\}_{\f{j}\in\mathbb{Z}^d}$ be the eigenvalues of $P_{\sqrt{2}\ep}$, see~\eqref{ev-prod}
  and~\eqref{eq:2005}. Let $\alpha\in(0,1)$ and $\beta\in(0,1)$ such that $\alpha+\beta\geq 1$, and let $s\geq 0$.
  \begin{enumerate}
  \item\label{lem:1i} The following bound holds
  \begin{align}
    \label{eq:1}
    \sum_{\f{j}\in\mathbb{Z}^d}{\{\lambda_{\f{j},\ep}\}(1+|\f{j}|^2)^{s}}\leq
    C(s,d)\left\{\ep^{-2\beta(2s+1)}+\ep^{-2\alpha(2s+1)}+\ep^{-2\alpha-4\beta s}\right\}\ep^{-(d-1)}.
  \end{align}
  \item\label{lem:1ii} The right-hand side of~\eqref{eq:1} is minimised, among all admissible pairs $(\alpha,\beta)$, by choosing
    $(\alpha,\beta)=(1/2,1/2)$. In this case, the right-hand side of~\eqref{eq:1} is proportional to $\ep^{-\theta_c(s)}$, where
    $\theta_c(s)=2s+d$ was introduced in Remark~\ref{rem:120}.
    \end{enumerate}
\end{lemma}

\begin{proof}[Proof of Lemma~\ref{lem:1} in dimension $d=1$]
We denote by $I_j(x)$ the $j$-th modified Bessel function of the first kind evaluated at $x$.

\emph{Step 1}. There exists $K>0$ such that, for any $j$ and any $\ep$, it holds $\lambda_{\ep,j}<K$.  This follows
from~\eqref{eq:2005} together with the monotonicity of $\{\lambda_{\ep,j}\}_j$ (see~\cite[Introduction]{Nasell1978a}).

\emph{Step 2}. Let $x\geq 1$. Picking $k=2$ and $m=0$ in~\cite[Theorem 2, bound (a)]{Nasell1978a}, we have
\begin{align}
  \label{eq:bes-mon}
  \frac{I_{j+1}(x)}{I_j(x)}<\frac{x}{j+1/2+x}.
\end{align}
We show that the inequality 
\begin{align}
  \label{eq:2}
  \frac{x}{j+1/2+x}\leq 1-\frac{1}{x^\alpha}
\end{align} 
holds when
\begin{align}
  j & \geq Cx^\beta,\label{eq:3}\\
  x & \geq \overline{x}=\overline{x}(\alpha,\beta)>0,\label{eq:4}
\end{align}
for suitable $\overline{x}(\alpha,\beta)>0$ and $C>0$ to be discussed below. Simple algebraic rearrangements imply
that~\eqref{eq:2} is equivalent to
\begin{align}
  \label{eq:5}
  0\leq j(x^{\alpha}-1)+\frac{1}{2}x^{\alpha}-x-\frac{1}{2},
\end{align}
which is in turn satisfied (taking~\eqref{eq:3} into account), at least under the sufficient condition
\begin{align}
  \label{eq:6}
  0\leq Cx^{\alpha+\beta}-Cx^{\beta}+\frac{1}{2}x^{\alpha}-x-\frac{1}{2}.
\end{align}
Take $C>0$ in~\eqref{eq:3} if $\alpha+\beta>1$, otherwise take $C>1$ if $\alpha+\beta=1$. Then, for $x$ large enough (i.e., for
$\overline{x}$ large enough in~\eqref{eq:4}), inequality~\eqref{eq:6} is satisfied, and therefore so is inequality~\eqref{eq:2}.

\emph{Step 3}. By symmetry of $\lambda_{\ep,j}$ with respect to $j$, seen in~\eqref{eq:2005}, we only need consider non-negative
indexes $j$. We define $A_1\coloneqq \{0,1,2,\dots,\lceil Cx^\beta\rceil\}$ and $A_2\coloneqq \mathbb{N}_0\setminus A_1$. We split
the sum in the left-hand side of~\eqref{eq:1} over these two sets. We use \emph{Step 1} to deduce
\begin{align}
  \label{eq:8}
  \sum_{j=0}^{\lceil Cx^\beta\rceil}{\{\lambda_{\ep,j}\}(1+j^2)^{s}} & \leq K\sum_{j=0}^{\lceil Cx^\beta\rceil}{(1+j^2)^{s}}
  \leq K\sum_{j=0}^{\lceil Cx^\beta\rceil}{(1+\lceil Cx^\beta\rceil^2)^{s}} \leq K\sum_{j=0}^{\lceil Cx^\beta\rceil}
       {(1+\lceil Cx^\beta\rceil)^{2s}}\nonumber\\
       & = K(1+\lceil Cx^\beta\rceil)^{2s+1} \leq C(s)K(\lceil Cx^\beta\rceil)^{2s+1}.
\end{align}
For the sum over $A_2$, we use the geometric decay $\lambda_{\ep,j+1}\leq (1-1/x^{\alpha})\lambda_{\ep,j}$, which is implied
by~\eqref{eq:bes-mon} and~\eqref{eq:2} combined with~\eqref{eq:2005}. We use \emph{Step 1} to obtain
\begin{align}
  & \sum_{j=\lceil Cx^\beta\rceil+1}^{\infty}{\{\lambda_{\ep,j}\}(1+j^2)^{s}} \nonumber\\
  & \quad \leq C(s)K\sum_{j=\lceil Cx^\beta\rceil+1}^{\infty}{\left(1-\frac{1}{x^{\alpha}}\right)^{j-(\lceil Cx^\beta\rceil+1)}j^{2s}}
  = C(s) K\sum_{j=0}^{\infty}{\left(1-\frac{1}{x^{\alpha}}\right)^{j}\left(j+\lceil Cx^\beta\rceil+1\right)^{2s}}\nonumber\\
  & \quad \leq C(s)K\sum_{j=0}^{\infty}{\left(1-\frac{1}{x^{\alpha}}\right)^{j}j^{2s}} + C(s)K\sum_{j=0}^{\infty}
  {\left(1-\frac{1}{x^{\alpha}}\right)^{j}\left(\lceil Cx^\beta\rceil+1\right)^{2s}}\nonumber\\
  & \quad\leq C(s)K\left(\frac{1}{x^{\alpha}}\right)^{-(2s+1)}+C(s)K x^{\alpha}\left(\lceil Cx^\beta\rceil+1\right)^{2s},\label{eq:7}
\end{align}
where we have also used estimates on the polylogarithmic function $\mbox{Li}_{\gamma}(z)\coloneqq
\sum_{j=1}^{\infty}{z^{j}j^{-\gamma}}$ for the first term in the last line, namely
\begin{align}
  \label{eq:17}
  \mbox{Li}_{\gamma}(z)\leq \frac{C(s)}{(1-z)^{-\gamma+1}}.
\end{align}
In our case, $\gamma=-2s$. Inequality~\eqref{eq:17} applies for negative integers $\gamma$ as a simple consequence of
differentiation of the geometric power series. Furthermore,~\eqref{eq:17} also applies for negative non-integers $\gamma$,
provided that $z\in(1-\nu,1)$ for some small $\nu=\nu(s)$. This is a consequence of the trivial bound $\mbox{Li}_{\gamma}(z)\leq
\sum_{j=0}^{\infty}{z^j(1+j)^{-\gamma}}$, and of~\cite[(9.550) and (9.557)]{Gradshteyn2007a}.

As in the case of~\cite[Lemma 4.3]{Cornalba2019a}, we pick $x\coloneqq (2\ep^2)^{-1}$, with $\ep$ small enough so
that~\eqref{eq:4} holds (and that $z=1-1/x^{\alpha}\in(1-\nu,1)$, with this requirement only demanded if $\gamma=-2s\notin
\mathbb{Z}$).  Combining~\eqref{eq:8} and~\eqref{eq:7} gives~\eqref{eq:1}.

Finally, it is easy to see that the choice $(\alpha,\beta)=(1/2,1/2)$, which makes the right-hand side of~\eqref{eq:1}
proportional to $\ep^{-\theta_c(s)}=\ep^{-(2s+d)}$, also minimises it among all admissible pairs $(\alpha,\beta)$.
\end{proof}

\begin{proof}[Proof of Lemma~\ref{lem:1} in dimension $d>1$]
  The result promptly follows from the bound
  \begin{align*}
    & \sum_{\f{j}\in\mathbb{Z}^d}{\{\lambda_{\f{j},\ep}\}(1+|\f{j}|^2)^{s}}
    = \sum_{\f{j}\in\mathbb{Z}^d}{\prod_{\ell=1}^{d}{\lambda_{j_{\ell},\ep}}\pp{1+\sum_{k=1}^{d}{j_k^2}}^{s}} \\
    & \quad \leq C(s)\sum_{k=1}^{d}{\sum_{\f{j}\in\mathbb{Z}^d}{\prod_{\ell=1}^{d}{\lambda_{j_{\ell},\ep}}(1+j_k^2)^{s}}}
    \leq C(s)\sum_{k=1}^{d}{\sum_{j_k\in\mathbb{Z}}{\lambda_{j_k,\ep}(1+j_k^2)^s}\sum_{j_\ell\in\mathbb{Z},\,\ell\neq k}
      {\prod_{\ell=1,\ell\neq k}^{d}{\lambda_{j_{\ell},\ep}}}}\\
    & \quad = C(s,d)\pp{\sum_{\smash{j\in\mathbb{Z}}}{\lambda_{j,\ep}(1+j^2)^s} }\left(\sum_{\smash{j\in\mathbb{Z}}} {\lambda_{j,\ep}}\right)^{d-1}
  \end{align*}
  and the validity of~\eqref{eq:1} for $d=1$. The optimality of the scaling under $(\alpha,\beta)=(1/2,1/2)$ has already been
  dealt with in the one-dimensional case.
\end{proof}

\begin{rem}
For $d=1$, we have improved the scaling of~\cite[Lemma 4.3]{Cornalba2019a} in two points. Firstly, the bound on
$\{\lambda_{j,\ep}\}_j$ is now uniform in $\ep$ and $j$ (i.e., we no longer bound $\lambda_{j,\ep}$ using $\ep^{-1}$). Secondly,
the exponential decay of the eigenvalues `kicks in' earlier, namely around $C\ep^{-2\beta}$ rather than around $\ep^{-2}$. This
leads to a sharper estimate concerning the sum on the region $A_1$.

These improvements bring the threshold $\theta_0$ down from $7$ to $3$ for the suboptimal choice $s=1$ (see~\cite[Lemma
  4.3]{Cornalba2019a}). In addition, the switch to fractional Sobolev spaces, i.e., the choice $s=1/2+\eta$ instead of $s=1$ as
in~\cite{Cornalba2019a}, where $\eta$ can be chosen arbitrarily small, grants a further decrease of $\theta_0$ from $3$ to $2$.
\end{rem}

\subsection{Regularity of the stochastic integrand \texorpdfstring{$B_{N,\delta}$}{}}
\label{ss:3}
\begin{lemma}
  \label{lem:7}
  With the same notation as in Theorem~\ref{thm:10}, let $s=d/2+\eta$, where $\eta>0$ is such that $\eta <C(d)<1/2$, where $C(d)$
  is small enough (see Lemma~\ref{gagl}). Then
  \begin{enumerate}
  \item\label{it:reg-1} $B_{N,\delta}$ is a map from $\mathcal{W}^s$ to $L(\mathcal{W}^s)$.
  \item\label{it:reg-2} $B_{N,\delta}$ is \emph{locally Lipschitz} with respect to the $L^0_2(\mathcal{W}^s)$-norm.
  \item\label{it:reg-3} $B_{N,\delta}$ is \emph{sublinear} with respect to the $L^0_2(\mathcal{W}^s)$-norm if $d=1$, and
    \emph{locally bounded} with respect to the same norm if $d>1$.
  \end{enumerate}
\end{lemma}

\begin{proof}[Proof of Lemma~\ref{lem:7} for $d=1$]
We limit ourselves to proving Statements~\ref{it:reg-2} and~\ref{it:reg-3}.

\emph{Statement~\ref{it:reg-2}}. Take $(u_1,v_1),(u_2,v_2)\in\mathcal{W}^s$, such that $\|(u_1,v_1)\|_{\mathcal{W}^s}\leq k,
\|(u_2,v_2)\|_{\mathcal{W}^s}\leq k$. From~\eqref{eq:unreg-noise} and~\eqref{eq:reg-noise} we have that
\begin{align}\label{eq:16}
  & \left\|B_{N,\delta}((u_1,v_1))-B_{N,\delta}((u_2,v_2))\right\|^2_{L^0_2(\mathcal{W}^s)} \nonumber\\
  &\quad = \sum_{j\in\mathbb{Z}}{\left\|\sqrt{\alpha_{j,s,\ep}}\left\{B_{N,\delta}((u_1,v_1))-B_{N,\delta}((u_2,v_2))\right\}(0,f_{j,s})
    \right\|^2_{\mathcal{W}^s}}\nonumber\\
  &\quad = \frac{\sigma^2}{N}\sum_{j\in\mathbb{Z}}{\alpha_{j,s,\ep}\left\|\left(0,\left\{h_{\delta}(u_1)-
    h_{\delta}(u_2)\right\}f_{j,s}\right)\right\|^2_{\mathcal{W}^s}}.
\end{align}
We use the fact that $\{f_{j,s}\}_j$ are orthonormal in $H^s$, the equivalence of the norms $\|\cdot\|_{H^s}$ and
$\|\cdot\|_{W^{s,2}}$ (see Subsection~\ref{ss:0}), the boundedness of $h'_{\delta}$, and Lemma~\ref{prod} to write
\begin{align}
  \label{eq:12}
  & \left\|\left(0,\left\{h_{\delta}(u_1)-
  h_{\delta}(u_2)\right\}f_{j,s}\right)\right\|^2_{\mathcal{W}^s}\nonumber\\
  & \quad \leq C\|\left\{h_{\delta}(u_1)-
  h_{\delta}(u_2)\right\}f_{j,s}\|_{W^{s,2}}^2
  \leq CK^2_{\ref{prod}}\|h_{\delta}(u_1)-
  h_{\delta}(u_2)\|_{W^{s,2}}^2\nonumber\\
  & \quad = CK^2_{\ref{prod}}\left\{\int_{\mathbb{T}}{|h_{\delta}(u_1(x))-h_{\delta}(u_2(x))|^2\m x} \right.\nonumber\\
  & \quad \quad \left.+\int_{\mathbb{T}}{\int_{\mathbb{T}}{\frac{|h_{\delta}(u_1(x))-h_{\delta}(u_2(x))-\left\{h_{\delta}(u_1(y))
        -h_{\delta}(u_2(y))\right\}|^2}{|x-y|^{1+2s}}\m x}\m y}\right\}\nonumber\\
  & \quad \leq C(\delta)K^2_{\ref{prod}}\left\{\int_{\mathbb{T}}{|u_1(x)-u_2(x)|^2\m x} \right.\nonumber\\
  & \quad \quad \left.+\int_{\mathbb{T}}{\int_{\mathbb{T}}{\frac{|h_{\delta}(u_1(x))-h_{\delta}(u_2(x))-\left\{h_{\delta}(u_1(y))
        -h_{\delta}(u_2(y))\right\}|^2}{|x-y|^{1+2s}}\m x}\m y}\right\}.
\end{align}
We bound the numerator of~\eqref{eq:12}. If either $u_1(x)=u_2(x)$ or $u_1(y)=u_2(y)$, then simply
\begin{align}
  \label{eq:13}
  & |h_{\delta}(u_1(x))-h_{\delta}(u_2(x))-\left\{h_{\delta}(u_1(y))-h_{\delta}(u_2(y))\right\}|^2\nonumber\\
  & \quad \leq C(\delta)|u_1(x)-u_2(x)-\left\{u_1(y)-u_2(y)\right\}|^2.
\end{align}
Otherwise, we use the embedding $H^s\subset C^0$ and write
\begin{align}
  \label{eq:14}
  & |h_{\delta}(u_1(x))-h_{\delta}(u_2(x))-\left\{h_{\delta}(u_1(y))-h_{\delta}(u_2(y))\right\}|^2\nonumber\\
  & \quad \leq 2\left|\frac{h_{\delta}(u_1(x))-h_{\delta}(u_2(x))}{u_1(x)-u_2(x)}\right|^2|u_1(x)-u_2(x)
  -\left\{u_1(y)-u_2(y)\right\}|^2\nonumber\\
  & \quad \quad +  2\left|\frac{h_{\delta}(u_1(x))-h_{\delta}(u_2(x))}{u_1(x)-u_2(x)}-\frac{h_{\delta}(u_1(y))
    -h_{\delta}(u_2(y))}{u_1(y)-u_2(y)}\right|^2\left|u_1(y)-u_2(y)\right|^2\nonumber\\
  & \quad \leq 2C(\delta)|u_1(x)-u_2(x)-\left\{u_1(y)-u_2(y)\right\}|^2\nonumber\\
  & \quad \quad +  2\K^2\left|\frac{h_{\delta}(u_1(x))-h_{\delta}(u_2(x))}{u_1(x)-u_2(x)}-\frac{h_{\delta}(u_1(y))
    -h_{\delta}(u_2(y))}{u_1(y)-u_2(y)}\right|^2\|u_1-u_2\|^2_{H^s}\nonumber\\
  & \quad =:T_1+T_2.
\end{align}
We now focus on $T_2$. We define the auxiliary function
\begin{align*}
  r(\alpha,\beta) =
  \begin{cases}
    \left\{h_{\delta}(\alpha)-h_{\delta}(\beta)\right\}/(\alpha-\beta), & \mbox{if } \alpha\neq \beta, \\
    h'_{\delta}(\alpha), & \mbox{if } \alpha=\beta.
  \end{cases}
\end{align*}
We write
\begin{align*}
  & \left|\frac{h_{\delta}(u_1(x))-h_{\delta}(u_2(x))}{u_1(x)-u_2(x)}-\frac{h_{\delta}(u_1(y))-h_{\delta}(u_2(y))}{u_1(y)-u_2(y)}\right|^2\\
  & \quad \leq 2\left|r(u_1(x),u_2(x))-r(u_1(y),u_2(x))\right|^2\\
  & \quad \quad+2\left|r(u_1(y),u_2(x))-r(u_1(y),u_2(y))\right|^2\coloneqq T_3+T_4.
\end{align*}
In the above, we perform a first-order Taylor expansion (with respect to the first variable of $r$ only for $T_3$, and with
respect to the second variable of $r$ only for $T_4$). This is possible because $r$ has partial derivatives defined everywhere (as
a consequence of $h_{\delta}$ being $C^2(\mathbb{R})$).  In addition, the partial derivatives of $r$ are uniformly bounded by
$\sup_{z\in\mathbb{R}}{|h''_{\delta}(z)|}\leq C(\delta)$. This implies
\begin{align}
  \label{eq:15}
  T_3+T_4 \leq C(\delta)|u_1(x)-u_1(y)|^2+C(\delta)|u_2(x)-u_2(y)|^2.
\end{align}
We plug~\eqref{eq:13},~\eqref{eq:14} and~\eqref{eq:15} into~\eqref{eq:12} and take into account the assumption
$\|(u_1,v_1)\|_{\mathcal{W}^s}\leq k, \|(u_2,v_2)\|_{\mathcal{W}^s}\leq k$ to obtain
\begin{align*}
  & \left\|\left(0,\left\{h_{\delta}(u_1)-
    h_{\delta}(u_2)\right\}f_{j,s}\right)\right\|^2_{\mathcal{W}^s} \nonumber\\
    & \quad \leq C(\delta)K^2_{\ref{prod}}\left\{\int_{\mathbb{T}}{|u_1(x)-u_2(x)|^2\m x} \right.\nonumber\\
    & \quad \quad \left.+\int_{\mathbb{T}}{\int_{\mathbb{T}}{\frac{|h_{\delta}(u_1(x))-h_{\delta}(u_2(x))
          -\left\{h_{\delta}(u_1(y))-h_{\delta}(u_2(y))\right\}|^2}{|x-y|^{1+2s}}\m x}\m y}\right\}\\
    & \quad \leq C(\delta)K^2_{\ref{prod}}\left\{\int_{\mathbb{T}}{|u_1(x)-u_2(x)|^2\m x} \right.\nonumber\\
    & \quad \quad \left.+\,C(\delta)\int_{\mathbb{T}}{\int_{\mathbb{T}}{\frac{|u_1(x)-u_2(x)-
          \left\{u_1(y)-u_2(y)\right\}|^2}{|x-y|^{1+2s}}\m x}\m y}\right.\nonumber\\
    & \quad \quad \left. +\,C(\delta)\K^2\|u_1-u_2\|^2_{H^s}\int_{\mathbb{T}}{\int_{\mathbb{T}}{\frac{|u_1(x)-u_1(y)|^2
          +|u_2(x)-u_2(y)|^2}{|x-y|^{1+2s}}\m x}\m y} \right\}\\
    & \quad \leq C(\delta,k,K_{\ref{prod}},\K)\|u_1-u_2\|^2_{H^s}\leq C(\delta,k,K_{\ref{prod}},\K)\|(u_1,v_1)-(u_2,v_2)\|_{\mathcal{W}^s}^2.
\end{align*}
We can go back to~\eqref{eq:16} and deduce the local Lipschitz property
\begin{align*}
  & \left\|B_{N,\delta}((u_1,v_1))-B_{N,\delta}((u_2,v_2))\right\|^2_{L^0_2(\mathcal{W}^s)} \\ 
  & \quad \leq \frac{\sigma^2}{N}\left(\sum_{j}{\alpha_{j,s,\ep}}\right)C(\delta,k,K_{\ref{prod}},\K)\|(u_1,v_1)-(u_2,v_2)\|_{\mathcal{W}^s}^2. 
\end{align*}
\emph{Statement~\ref{it:reg-3}}. We write
\begin{align}
  \label{eq:18}
  & \left\|B_{N,\delta}((u,v))\right\|^2_{L^0_2(\mathcal{W}^s)}\nonumber \\
  & \quad = \sum_{j\in\mathbb{Z}}{\left\|\sqrt{\alpha_{j,s,\ep}}B_{N,\delta}((u,v))(0,f_{j,s})\right\|^2_{\mathcal{W}^s}}
    =\frac{\sigma^2}{N}\sum_{j\in\mathbb{Z}}{\alpha_{j,s,\ep}\left\|\left(0,h_{\delta}(u)f_{j,s}\right)\right\|^2_{\mathcal{W}^s}}\nonumber\\
  & \quad = \frac{\sigma^2}{N}\left[\sum_{j\in\mathbb{Z}}{\alpha_{j,s,\ep}\left\|{h_{\delta}(u)f_{j,s}}\right\|^2_{H^s}}
    \right] \leq K^2_{\ref{prod}}\frac{\sigma^2}{N}\left[\sum_{j\in\mathbb{Z}}{\alpha_{j,s,\ep}\left\|{h_{\delta}(u)}\right\|^2_{H^s}}
    \right]\nonumber\\
   & \quad     \leq C(\delta)K^2_{\ref{prod}}\frac{\sigma^2}{N}\ep^{-(2s+1)}\left(1+\|(u,v)\|^2_{\mathcal{W}^s}\right),
\end{align}
where we have used Lemma~\ref{prod}, the sublinearity of $h_{\delta}$ at infinity, the boundedness of $h'_{\delta}$, and
Lemma~\ref{lem:1}. This completes the proof.
\end{proof}

\begin{proof}[Proof of Lemma~\ref{lem:7} for $d>1$]

In this proof, we need to analyse quantities associated with derivatives of the distinctive nonlinearity $h_{\delta}(u)$, $u\in
H^{d/2+\eta}$. For this purpose, we make heavy use of the contents of Appendix~\ref{app:b} (integrability properties of the Fa\`a
di Bruno representation of derivatives of $h_{\delta}(u)$) and Appendix~\ref{app:a} (factorisation of differences of two distinct
instances of the same derivative).

We again focus on points~\ref{it:reg-2} and~\ref{it:reg-3} only.

\emph{Statement~\ref{it:reg-2}}. Take $(u_1,\f{v}_1),(u_2,\f{v}_2)\in\mathcal{W}^s$, such that
$\|(u_1,\f{v}_1)\|_{\mathcal{W}^s}\leq k, \|(u_2,\f{v}_2)\|_{\mathcal{W}^s}\leq k$. In order to bound
\begin{equation*}
  \left\|B_{N,\delta}((u_1,\f{v}_1))-B_{N,\delta}((u_2,\f{v}_2))\right\|^2_{L^0_2(\mathcal{W}^s)}
\end{equation*}
we only need to bound
\begin{align*}
  \frac{\sigma^2}{N}\sum_{\f{j}\in\mathbb{Z}}{\alpha_{\f{j},s,\ep}\left\|\left(0,\left\{h_{\delta}(u_1)-
    h_{\delta}(u_2)\right\}f_{\f{j},s}\right)\right\|^2_{H^s\times H^s}}.
\end{align*}
Moreover, Lemma~\ref{prod} allows us to only focus on estimating $\|h_{\delta}(u_1)- h_{\delta}(u_2)\|^2_{H^s}$. We introduce the
shorthand notations
\begin{align*}
  \mathcal{P}_{\pi,\alpha}u(\f{x})\coloneqq \prod_{j\in J(\pi)}{\prod_{b\in B(\pi)\colon |b|=j}{\frac{\partial^{(j)}u(\f{x})}
      {\prod_{z\in b}{\partial x_{\ell_z}}}}},\qquad \mathcal{P}_{\pi}u\coloneqq \mathcal{P}_{\pi,\lfloor d/2 \rfloor}u
\end{align*}
for every $\mathbb{N}\ni\alpha\leq \lfloor d/2 \rfloor$ and $\pi\in\Pi_{\alpha}$, and $\mathcal{P}_{\pi,0}u(\f{x})\coloneqq 1$.
Due to the Fa\`a di Bruno formula recalled in Lemma~\ref{faa} and the equivalence of the norms $\|\cdot\|_{H^s}$ and
$\|\cdot\|_{W^{s,2}}$, the term $\|h_{\delta}(u_1)- h_{\delta}(u_2)\|^2_{H^s}$ can be controlled by providing a bound for
\begin{align}
  \label{eq:4000}
  & \left[\sum_{\alpha=0}^{\lfloor d/2 \rfloor}{\left\| h_{\delta}^{(|\pi_{\alpha}|)}(u_1)\mathcal{P}_{\pi_{\alpha},\alpha}u_1
      - h_{\delta}^{(|\pi_{\alpha}|)}(u_2)\mathcal{P}_{\pi_{\alpha},\alpha}u_2\right\|^2_{L^2}}\right]\nonumber\\
  & \quad +\int_{\mathbb{T}^d}{\int_{\mathbb{T}^d}{\frac{1}{|\f{x}-\f{y}|^{d+(s-\lfloor s\rfloor)2}}
      \left|h_{\delta}^{(|\pi|)}(u_1(\f{x}))\mathcal{P}_{\pi}u_1(\f{x})-h_{\delta}^{(|\pi|)}(u_2(\f{x}))\mathcal{P}_{\pi}u_2(\f{x})\right.} }
  +\nonumber\\
  & \quad \quad \left.-\left\{h_{\delta}^{(|\pi|)}(u_1(\f{y}))\mathcal{P}_{\pi}u_1(\f{y})-h_{\delta}^{(|\pi|)}(u_2(\f{y}))
  \mathcal{P}_{\pi}u_2(\f{y})\right\}\right|^2\m\f{x}\m\f{y} \coloneqq  A_1 + A_2
\end{align}
for any choice $\pi_{\alpha}\in\Pi_{\alpha}$, $\alpha\in\{0,\dots,\lfloor d/2 \rfloor\}$, and $\pi\in\Pi_{\lfloor d/2 \rfloor}$.

Upon adding and subtracting terms of the type $h_{\delta}^{(|\pi_{\alpha}|)}(u_2)\mathcal{P}_{\pi_{\alpha},\alpha}u_1$, for
$\alpha\in\{0,\dots,\lfloor d/2 \rfloor\}$, the term $A_1$ is bounded (up to a constant) by
\begin{align}
  &\sum_{\alpha=0}^{\lfloor d/2 \rfloor}{\left\| \left\{h_{\delta}^{(|\pi_{\alpha}|)}(u_1) -
    h_{\delta}^{(|\pi_{\alpha}|)}(u_2)\right\}\mathcal{P}_{\pi_{\alpha},\alpha}u_2\right\|^2_{L^2}}
  + \sum_{\alpha=0}^{\lfloor d/2 \rfloor}{\left\| h_{\delta}^{(|\pi_{\alpha}|)}(u_2)\left\{\mathcal{P}_{\pi_{\alpha},\alpha}u_1
    - \mathcal{P}_{\pi_{\alpha},\alpha}u_2\right\}\right\|^2_{L^2}}\nonumber\\
  & \quad \label{eq:6000}\leq C(d,\delta)\left[\K^2\|u_1-u_2\|^2_{H^s}\sum_{\alpha=0}^{\lfloor d/2 \rfloor}{\left\|
      \mathcal{P}_{\pi_{\alpha},\alpha}u_2\right\|^2_{L^2}} + \sum_{\alpha=0}^{\lfloor d/2 \rfloor}{\left\| \mathcal{P}_{\pi_{\alpha},\alpha}u_1
      - \mathcal{P}_{\pi_{\alpha},\alpha}u_2\right\|^2_{L^2}}\right]\\
  & \quad \label{eq:6001}\leq C(d,\delta,k)\K^2K_{\ref{der}}\|u_1-u_2\|^2_{H^s}
  +C(d,\delta)\sum_{\alpha=0}^{\lfloor d/2 \rfloor}{\left\| \mathcal{P}_{\pi_{\alpha},\alpha}u_1
    - \mathcal{P}_{\pi_{\alpha},\alpha}u_2\right\|^2_{L^2}},
\end{align}
where we have used a Taylor expansion for (and boundedness of) derivatives of $h_{\delta}$ and the Sobolev embedding $H^s\subset
C^0$ in~\eqref{eq:6000}, and Lemma~\ref{der} in~\eqref{eq:6001}. We may now apply Lemma~\ref{fac}--\ref{fac:i} to factorise
$\mathcal{P}_{\pi_{\alpha},\alpha}u_1 - \mathcal{P}_{\pi_{\alpha},\alpha}u_2$ into a sum of terms, each of which can then be dealt
with using Lemma~\ref{der}. We obtain
\begin{align}
  \label{eq:6002}
  A_1 \leq \eqref{eq:6001} \leq C(d,\delta,k)\K^2K_{\ref{der}}\|u_1-u_2\|^2_{H^s}.
\end{align}
More generally, each application of Lemma~\ref{fac} below is, at least conceptually, identical to the one illustrated
above. Namely, it is used to factorise a difference of objects into a sum of terms which in turn can be estimated using either
Lemma~\ref{der} or Lemma~\ref{gagl}.

Following simple algebraic rewritings, the term $A_2$ can be bounded (up to a constant) by
\begin{align}
  \label{eq:4001}
  & \int_{\mathbb{T}^d}{\int_{\mathbb{T}^d}{\frac{ \left|h_{\delta}^{(|\pi|)}(u_1(\f{x}))-h_{\delta}^{(|\pi|)}(u_2(\f{x}))
        -\{h_{\delta}^{(|\pi|)}(u_1(\f{y}))
        -h_{\delta}^{(|\pi|)}(u_2(\f{y}))\}\right|^2\left|\mathcal{P}_{\pi}u_1(\f{x})\right|^2 }
      {|\f{x}-\f{y}|^{d+(s-\lfloor s\rfloor)2}}}\m\f{x}\m\f{y}}\nonumber\\
  & \quad + \int_{\mathbb{T}^d}{\int_{\mathbb{T}^d}{\frac{ \left|h_{\delta}^{(|\pi|)}(u_1(\f{y}))
        -h_{\delta}^{(|\pi|)}(u_2(\f{y}))\right|^2\left|\mathcal{P}_{\pi}u_1(\f{x})-\mathcal{P}_{\pi}u_1(\f{y})\right|^2 }{|\f{x}
        -\f{y}|^{d+(s-\lfloor s\rfloor)2}}}\m\f{x}\m\f{y}}\nonumber\\
  &\quad + \int_{\mathbb{T}^d}{\int_{\mathbb{T}^d}{\frac{ \left|h_{\delta}^{(|\pi|)}(u_2(\f{x}))\right|^2
        \left|\mathcal{P}_{\pi}u_1(\f{x})-\mathcal{P}_{\pi}u_2(\f{x})
        -\{\mathcal{P}_{\pi}u_1(\f{y})-\mathcal{P}_{\pi}u_2(\f{y})\}\right|^2 }
      {|\f{x}-\f{y}|^{d+(s-\lfloor s\rfloor)2}}}\m\f{x}\m\f{y}}\nonumber\\
  & \quad + \int_{\mathbb{T}^d}{\int_{\mathbb{T}^d}{\frac{  \left|h_{\delta}^{(|\pi|)}(u_2(\f{x}))
        -h_{\delta}^{(|\pi|)}(u_2(\f{y}))\right|^2\left|\mathcal{P}_{\pi}u_1(\f{y})-\mathcal{P}_{\pi}u_2(\f{y})\right|^2 }
      {|\f{x}-\f{y}|^{d+(s-\lfloor s\rfloor)2}}}\m\f{x}\m\f{y}}\nonumber\\
  & \quad \quad =:T_1+\dots+T_4.
\end{align}     
Term $T_1$ is dealt with using~\eqref{eq:14} and~\eqref{eq:15} (with $h^{(|\pi|)}_{\delta}$ replacing $h_{\delta}$), the embedding
$H^s\subset C^0$, and Lemma~\ref{gagl}--\ref{gagl:i}. Its bound reads
\begin{align*}
T_1 & \leq C(\delta)\int_{\mathbb{T}^d}{\int_{\mathbb{T}^d}{\frac{ \left|u_1(\f{x})-u_2(\f{x})-\{u_1(\f{y})
  -u_2(\f{y})\}\right|^2\left|\mathcal{P}_{\pi}u_1(\f{x})\right|^2 }{|\f{x}-\f{y}|^{d+(s-\lfloor s\rfloor)2}}}\m\f{x}\m\f{y}}\\
  & \quad + C(\delta)\K^2\|u_1-u_2\|_{H^s}^2\int_{\mathbb{T}^d}{\int_{\mathbb{T}^d}{\frac{ \left|u_1(\f{x})-u_1(\f{y})
  \right|^2\left|\mathcal{P}_{\pi}u_1(\f{x})\right|^2 }{|\f{x}-\f{y}|^{d+(s-\lfloor s\rfloor)2}}}\m\f{x}\m\f{y}}\\
  & \quad + C(\delta)\K^2\|u_1-u_2\|_{H^s}^2\int_{\mathbb{T}^d}{\int_{\mathbb{T}^d}{\frac{ \left|u_2(\f{x})-u_2(\f{y})
  \right|^2\left|\mathcal{P}_{\pi}u_1(\f{x})\right|^2 }{|\f{x}-\f{y}|^{d+(s-\lfloor s\rfloor)2}}}\m\f{x}\m\f{y}}\\
  & \leq C(\delta,k)\K^2K_{\ref{gagl}}\|u_1-u_2\|_{H^s}^2.
\end{align*}
The embedding $H^s\subset C^0$ and Lemmas~\ref{fac}--\ref{fac:i} and~\ref{gagl}--\ref{gagl:ii} allow to bound $T_2$ as
\begin{align*}
  T_2 & \leq C(\delta)\K^2\|u_1-u_2\|_{H^s}^2\int_{\mathbb{T}^d}{\int_{\mathbb{T}^d}{\frac{ \left|\mathcal{P}_{\pi}u_1(\f{x})
        -\mathcal{P}_{\pi}u_1(\f{y})\right|^2 }{|\f{x}-\f{y}|^{d+(s-\lfloor s\rfloor)2}}}\m\f{x}\m\f{y}}\\
  & \leq  C(\delta,k)\K^2K_{\ref{gagl}}\|u_1-u_2\|_{H^s}^2.
\end{align*}
Term $T_3$ is dealt with by relying on the boundedness of $h^{(|\pi|)}_{\delta}$ and using Lemmas~\ref{fac}--\ref{fac:ii}
and~\ref{gagl}--\ref{gagl:ii}, thus giving
\begin{align*}
  T_3 & \leq C(\delta)\int_{\mathbb{T}^d}{\int_{\mathbb{T}^d}{\frac{ \left|\mathcal{P}_{\pi}u_1(\f{x})-\mathcal{P}_{\pi}u_2(\f{x})
        -\{\mathcal{P}_{\pi}u_1(\f{y})-\mathcal{P}_{\pi}u_2(\f{y})\}\right|^2 }{|\f{x}-\f{y}|^{d+(s-\lfloor s\rfloor)2}}}\m\f{x}\m\f{y}}\\
  & \leq C(\delta,k)K_{\ref{gagl}}\|u_1-u_2\|_{H^s}^2.
\end{align*}
Finally, term $T_4$ is dealt with using a Taylor expansion on $h^{(|\pi|)}_{\delta}$, and Lemmas~\ref{fac}--\ref{fac:i}
and~\ref{gagl}--\ref{gagl:i}. Its bounds reads
\begin{align*}
T_4 & \leq  C(\delta)\int_{\mathbb{T}^d}{\int_{\mathbb{T}^d}{\frac{  \left|u_2(\f{x})
      -u_2(\f{y})\right|^2\left|\mathcal{P}_{\pi}u_1(\f{y})-\mathcal{P}_{\pi}u_2(\f{y})\right|^2 }
    {|\f{x}-\f{y}|^{d+(s-\lfloor s\rfloor)2}}}\m\f{x}\m\f{y}}\\
  & \leq C(\delta,k)K_{\ref{gagl}}\|u_1-u_2\|_{H^s}^2.
\end{align*}
Putting together the bounds obtained for~\eqref{eq:4001} and~\eqref{eq:6002} into~\eqref{eq:4000}, and using Lemma~\ref{prod} and
Lemma~\ref{lem:1}--\ref{lem:1ii}, we deduce
\begin{align}
  \label{eq:31}
  & \left\|B_{N,\delta}((u_1,\f{v}_1))-B_{N,\delta}((u_2,\f{v}_2))\right\|^2_{L^0_2(\mathcal{W}^s)}\nonumber\\
  & \quad \leq K^2_{\ref{prod}}C(\delta,k,d,\K,K_{\ref{der}},K_{\ref{gagl}})\sigma^2 N^{-1}\ep^{-(2s+d)}
  \left\|(u_1,\f{v}_1)-(u_2,\f{v}_2)\right\|^2_{\mathcal{W}^s}.
\end{align}

\emph{Statement~\ref{it:reg-3}}. The proof is similar to that of Statement~\ref{it:reg-2}. Take $(u,\f{v})\in\mathcal{W}^s$, such
that $\|(u,\f{v})\|_{\mathcal{W}^s}\leq k$. We only need to bound
\begin{align}
  \label{eq:4002}
  & \left[\sum_{\alpha=0}^{\lfloor d/2 \rfloor}{\left\| h_{\delta}^{(|\pi_{\alpha}|)}(u)
      \mathcal{P}_{\pi_{\alpha},\alpha}u \right\|^2_{L^2}}\right]\nonumber\\
  & \quad +\int_{\mathbb{T}^d}{\int_{\mathbb{T}^d}{\frac{\left|h_{\delta}^{(|\pi|)}(u(\f{x}))
        \mathcal{P}_{\pi}u(\f{x})-h_{\delta}^{(|\pi|)}(u(\f{y}))\mathcal{P}_{\pi}u(\f{y})\right|^2}{|\f{x}-\f{y}|^{d+(s-\lfloor s\rfloor)2}}
      \m\f{x}} \m \f{y}}\coloneqq  A_3 + A_4
\end{align}
for any choice $\pi_{\alpha}\in\Pi_{\alpha}$, $\alpha\in\{0,\dots,\lfloor d/2 \rfloor\}$, and $\pi\in\Pi_{\lfloor d/2
  \rfloor}$. The term $A_3$ is easily settled using the boundedness of derivatives of $h_{\delta}$ and
Lemma~\ref{der}. Furthermore, $A_4$ is bounded (up to a constant) by
\begin{align}
  \label{eq:4003}
  & \int_{\mathbb{T}^d}{\int_{\mathbb{T}^d}{  \frac{\left|h_{\delta}^{(|\pi|)}(u(\f{x}))
        -h_{\delta}^{(|\pi|)}(u(\f{y}))\right|^2\left|\mathcal{P}_{\pi}u(\f{y})\right|^2}
      {|\f{x}-\f{y}|^{d+(s-\lfloor s\rfloor)2}}\m \f{x}\m \f{y} }}\nonumber\\
  & \quad + \int_{\mathbb{T}^d}{\int_{\mathbb{T}^d}{  \frac{\left|h_{\delta}^{(|\pi|)}(u(\f{x}))\right|^2\left|\mathcal{P}_{\pi}
        u(\f{x})-\mathcal{P}_{\pi}u(\f{y})\right|^2}{|\f{x}-\f{y}|^{d+(s-\lfloor s\rfloor)2}}\m \f{x}\m \f{y} }}\coloneqq T_5+T_6 .
\end{align}
Term $T_5$ is bounded using a Taylor expansion of $h^{(\pi)}_{\delta}$, and Lemma~\ref{gagl}--\ref{gagl:i}. Term $T_6$ is bounded
relying on the boundedness of $h^{(|\pi|)}_{\delta}$ and using Lemmas~\ref{fac}--\ref{fac:i} and~\ref{gagl}--\ref{gagl:ii}.

Putting the bounds obtained for~\eqref{eq:4003} into~\eqref{eq:4002} and using Lemmas~\ref{prod} and~\ref{lem:1}--\ref{lem:1ii},
we deduce
\begin{align}\label{eq:6003}
  \left\|B_{N,\delta}((u,\f{v}))\right\|^2_{L^0_2(\mathcal{W}^s)} & = \frac{\sigma^2}{N}\sum_{\f{j}\in\mathbb{Z}^d}{\alpha_{\f{j},s,\ep}\left\|
    \left(0,h_{\delta}(u)f_{\f{j},s}\right)\right\|^2_{\mathcal{W}^s}}
  \leq C(d)K^2_{\ref{prod}}\frac{\sigma^2}{N}\sum_{\f{j}\in\mathbb{Z}^d}{\alpha_{\f{j},s,\ep}\left\|{h_{\delta}(u)}\right\|^2_{H^s}}
   \nonumber\\
  &  \leq K^2_{\ref{prod}}C(\delta,d,K_{\ref{der}},K_{\ref{gagl}})\frac{\sigma^2}{N}\sum_{\f{j}\in\mathbb{Z}^d}{\alpha_{\f{j},s,\ep}}
   \left(1+\|(u,\f{v})\|^{2(\lfloor d/2\rfloor +1)}_{\mathcal{W}^s}\right)\nonumber\\
   & \leq K^2_{\ref{prod}}C(\delta,s,d,K_{\ref{der}},K_{\ref{gagl}})\frac{\sigma^2}{N}\ep^{-(2s+d)}
   \left(1+\|(u,\f{v})\|^{2(\lfloor d/2\rfloor +1)}_{\mathcal{W}^s}\right).
\end{align}
The assumption $\|(u,\f{v})\|_{\mathcal{W}^s}\leq k$ gives the desired local boundedness property. The proof is complete.
\end{proof}

\section{Proof of Theorem~\ref{thm:10}}
\label{ss:20}
This is an adaptation of~\cite[Theorem 4.4]{Cornalba2020a}, and we heavily rely on the tools developed in Section~\ref{s:2}.  The
functional $\alpha_{U}$ is locally Lipschitz and locally bounded in the $\mathcal{W}^s$-norm. This is a consequence of the
following simple bound for $u\in H^s$ and $\ell\in\{1,\dots,d\}$
\begin{align*}
  \|\partial_{x_{\ell}} U\ast u\|^2_{H^s} & = \sum_{\f{j}\in\mathbb{Z}^d}{\widehat{(\partial_{x_{\ell}}U\ast u)}_{\f{j}}
    \overline{\widehat{(\partial_{x_{\ell}}U\ast u)}}_{\f{j}}(1+|\f{j}|^2)^s}=\sum_{\f{j}
    \in\mathbb{Z}^d}{\left|\widehat{(\partial_{x_{\ell}}U)}_{\f{j}}\right|^2\left|\widehat{u}_{\f{j}}\right|^2(1+|\f{j}|^2)^s}\\
  & \leq C(\|U\|_{C^1},d)\sum_{\f{j}\in\mathbb{Z}^d}{\left|\widehat{u}_{\f{j}}\right|^2(1+|\f{j}|^2)^s}=C(\|U\|_{C^1},d)\|u\|^2_{H^s}.
\end{align*}
These properties of $\alpha_{U}$, together with Lemmas~\ref{lem:80} and~\ref{lem:7}, allow us to use~\cite[Theorem
  4.5]{Tappe2012a} and deduce the existence and uniqueness of a local $\mathcal{W}^s$-valued mild solution to~\eqref{eq:25} in the
sense of~\cite[Chapter 7]{Da-Prato2014a}. Specifically, there is a stopping time $\tau>0$ and a unique $\mathcal{W}^s$-valued
predictable process $X_{\ep,\delta}=(\tilde{\rho}_{\ep,\delta},\tildebj_{\ep,\delta})$ defined on $[0,\tau]$ such that
$\mathbb{P}(\int_{0}^{\tau}{\|X_{\ep,\delta}(z)\|^2_{\mathcal{W}^s}\m z}<\infty)=1$, and satisfying, for each $t>0$
\begin{align}
  \label{eq:2004}
  X_{\ep,\delta}(t\wedge \tau) & = S(t\wedge \tau)\,X_0+\int_{0}^{t\wedge \tau}{S(t\wedge \tau-s)\,\alpha_{U}(X_{\ep,\delta}(s))\,\m s}\nonumber\\
  & \quad +\int_{0}^{t\wedge \tau}{S(t\wedge \tau-s)\,B_{N,\delta}(X_{\ep,\delta}(s))\,\m W_{\ep}},\qquad\mathbb{P}\mbox{-a.s.},
\end{align}
where $\{S(t)\}_{t\geq 0}$ is the $C_0$-semigroup generated by $A$. Using~\cite[Theorem 4.5 and Remark 4.6]{Tappe2012a}, the
continuous embedding $H^s\subset C^0$, and the assumption $\min_{\f{x}\in\mathbb{T}^d}\tilde{\rho}_{0}(\f{x})> \delta$, we deduce
that there exists $T=T(\tilde{\rho}_0)$ and a unique deterministic $\mathcal{W}^s$-valued mild solution
$Z_{\delta}=(\rho_Z,\f{j}_Z)$ to the noise-free equivalent of~\eqref{eq:25} up to $T$, such that
$\min_{\f{x}\in\mathbb{T}^d,s\in[0,T]}\rho_Z(\f{x},s)>\delta$. It is also obvious that there is $k>0$ such that
$\max_{s\in[0,T]}\|\rho_Z(\cdot,s)\|_{\mathcal{W}^s}<k$.
 
We compare $X_{\ep,\delta}$ and $Z_{\delta}$. As $X_{\ep,\delta}$ is a local mild solution, it is well-defined up to the first
exit time from the $\mathcal{W}^s$-ball of radius $k$. In particular, $X_{\ep,\delta}$ and $Z_{\delta}$ are well-defined up to the
stopping time
\begin{align}
  \label{eq:2003}
  \tau_{\delta, k} & \coloneqq  \inf\Big\{t>0:\|X_{\ep,\delta}(t)\|_{\mathcal{W}^s}\geq k\Big\}
  \wedge \inf\Big\{t>0:\min_{\f{x}\in\mathbb{T}^d}{\tilde{\rho}_{\ep,\delta}(\f{x},t)}\leq \delta\Big\} \wedge T.
\end{align}
We consider the difference
\begin{align}
  \label{i:305}
  X_{\ep,\delta}(t\wedge\tau_{\delta,k})-Z(t\wedge\tau_{\delta,k}) & =\int_{0}^{t\wedge\tau_{\delta,k}}
  {S(t\wedge \tau_{\delta,k}-s)\Big[\alpha_{U}(X_{\ep,\delta}(s))-\alpha_{U}(Z(s))\Big]\,\m s}\nonumber\\
  & \quad +\int_{0}^{t\wedge\tau_{\delta,k}}{S(t\wedge \tau_{\delta,k}-s)\,B_{N,\delta}(X_{\ep,\delta}(s))\,\m W_{\ep}}.
\end{align}
Let $q>2$. For some $C_1=C_1\left(U,k,T,q,\eta,d,K_{\ref{prod}}\right)$, we have
\begin{align}\label{eq:7000}
  & \mean{\sup_{s\in[0,t]}{\left\|X_{\ep,\delta}(s\wedge\tau_{\delta,k})-Z(s\wedge\tau_{\delta,k})\right\|^q_{\mathcal{W}^s}}} \nonumber\\
  & \quad \leq C_1\,\mean{\int_{0}^{t}{\left\|X_{\ep,\delta}(u)-Z(u)\right\|^q_{\mathcal{W}^s}\,\mathbf{1}_{[0,\tau_{\delta,k}]}(u)\,\m u}}\nonumber\\
  & \quad\quad +\mean{\sup_{s\in[0,T]}{\left\|\int_{0}^{s}{S(s\wedge\tau_{\delta,k}-u)B_{N,\delta}
        (X_{\ep,\delta}(u))\,\mathbf{1}_{[0,\tau_{\delta,k}]}(u)\,\m W_{\ep}}\right\|_{\mathcal{W}^s}^q}}. 
        \end{align}
        We use~\cite[Proposition 7.3]{Da-Prato2014a} and Lemma~\ref{lem:7}-\ref{it:reg-3}, inequality~\eqref{eq:6003}, to provide the bound
        \begin{align}
   \eqref{eq:7000} & \leq C_1\!\int_{0}^{t}{\mean{\sup_{s\in[0,u]}\left\|X_{\ep,\delta}(s\wedge\tau_{\delta,k})
       -Z(s\wedge\tau_{\delta,k})\right\|^q_{\mathcal{W}^s}}\m u} + C(\sigma,\delta,T,q,\eta,d,K_{\ref{prod}},K_{\ref{der}},K_{\ref{gagl}})
   \nonumber\\
  & \quad \times \left(N^{-1}\ep^{-(2s+d)}\right)^{q/2}
  \mean{\int_{0}^{T}{\!\!\left(1+\|X_{\ep,\delta}(u)\|^{(\lfloor d/2 \rfloor +1)q}_{\mathcal{W}^s}\right)
      \mathbf{1}_{[0,\tau_{\delta,k}]}(u)\,\m u}}\nonumber\\
  & \leq C_1\!\int_{0}^{t}{\mean{\sup_{s\in[0,u]}\left\|X_{\ep,\delta}(s\wedge\tau_{\delta,k})
      -Z(s\wedge\tau_{\delta,k})\right\|^q_{\mathcal{W}^s}}\m u}
  +C_2\left(N^{-1}\ep^{-(2s+d)}\right)^{q/2},
  \label{i:306}
\end{align}
for some $C_2=C_2(\sigma,\delta,T,q,k,\eta,d,K_{\ref{prod}},K_{\ref{der}},K_{\ref{gagl}})$. Crucially, the last inequality is not
affected by the superlinear nature of the noise for $d>1$, as $X_{\ep,\delta}$ lives on a bounded set of $\mathcal{W}^s$ up to
$\tau_{\delta,k}$. Applying the Gronwall Lemma to~\eqref{eq:7000}--\eqref{i:306} gives
\begin{align}
  \label{i:307}
  \mean{\sup_{s\in[0,T]}{\left\|X_{\ep,\delta}(s\wedge\tau_{\delta,k})-Z(s\wedge\tau_{\delta,k})\right\|^q_{\mathcal{W}^s}}} \leq
  C_2\left(N^{-1}\ep^{-(2s+d)}\right)^{q/2}e^{TC_1}.
\end{align}
The choice of $\theta$ for~\eqref{eq:2000} given in the assumption and a Chebyshev-type argument imply that 
\begin{equation*}
  \lim_{N\rightarrow \infty}\mathbb{P}\left(\sup_{s\in[0,T]}{\left\|X_{\ep,\delta}(s\wedge\tau_{\delta,k})
    -Z(s\wedge\tau_{\delta,k})\right\|_{\mathcal{W}^s}}\geq \beta \right)=1
\end{equation*}
for any $\beta\in(0,1)$. It is now a standard routine (see~\cite[Theorem 4.4]{Cornalba2020a}) to pick $\beta$ small enough, $N$
big enough, and deduce the existence of a set $F_{\nu}$ such that $\mathbb{P}(F_{\nu})>1-\nu$, on which $\tau_{\delta,k}\equiv T$,
on which $\tilde{\rho}_{\ep}\geq \delta$, and on which~\eqref{eq:25} is satisfied by $X_{\ep,\delta}$ the sense of mild
solutions. Going back to~\eqref{eq:2004}, this implies
\begin{align*}
  X_{\ep,\delta}(t) & = S(t)\,X_0+\int_{0}^{t}{S(t-s)\,\alpha_{U}(X_{\ep,\delta}(s))\,\m s}+\int_{0}^{t}{S(t-s)\,
    B_{N,\delta}(X_{\ep,\delta}(s))\,\m W_{\ep}}\\
  & = S(t)\,X_0+\int_{0}^{t}{S(t-s)\,\alpha_{U}(X_{\ep,\delta}(s))\,\m s}+\int_{0}^{t}{S(t-s)\,B_{N}(X_{\ep,\delta}(s))
    \,\m W_{\ep}}\quad\mbox{on }F_{\nu},
 \end{align*}
and this concludes the proof.

\appendix

\section{Factorisation of products}\label{app:a}
We recall the following simple factorisation for differences of products.
\begin{lemma}\label{fac}
  Let $\f{a},\f{b},\f{c},\f{d}\in\mathbb{R}^N$.
  \begin{enumerate}
  \item\label{fac:i} We have 
    \begin{align}
      \label{i:10}
      \prod_{i=1}^{N}{a_i}-\prod_{i=1}^{N}{b_i}=\sum_{k=1}^{N}{b_{<k}(a_k-b_k)a_{>k}},
    \end{align}
    where we have used the shorthand notations $b_{<k}\coloneqq \prod_{j=1}^{k-1}{b_j}$ and $a_{>k}\coloneqq \prod_{j=k+1}^{N}{a_j}$ (with the
    usual convention of the product over an empty set being unitary).
  \item\label{fac:ii} For each $k=1,\dots,N$, consider the families 
    \begin{align*}
      & \{\alpha^k_j\}_{j=1}^{N-1}\coloneqq (b_1,\dots,b_{k-1},a_{k+1},\cdots,a_N),
      & \{\beta^k_j\}_{j=1}^{N-1}\coloneqq (d_1,\dots,d_{k-1},c_{k+1},\cdots,c_N).
    \end{align*}
    We have
    \begin{align}
      \label{i:11}
      \prod_{i=1}^{N}{a_i}-\prod_{i=1}^{N}{b_i}-\left(\prod_{i=1}^{N}{c_i}-\prod_{i=1}^{N}{d_i}\right) &=
      \sum_{k=1}^{N}{b_{<k}(a_k-b_k-(c_k-d_k))a_{>k}}\nonumber\\
      & \quad +\sum_{k=1}^{N}{\sum_{j=1}^{N-1}{(c_k-d_k)(\beta^k_{<j}(\alpha^k_j-\beta^k_j)\alpha^k_{>j})}}.
    \end{align}
  \end{enumerate}
\end{lemma}

\begin{proof}
Point~\ref{fac:i} is easily proven by induction. As for Point~\ref{fac:ii}, we use Point~\ref{fac:i} twice and obtain
\begin{align*}
  & \prod_{i=1}^{N}{a_i}-\prod_{i=1}^{N}{b_i}-\left(\prod_{i=1}^{N}{c_i}-\prod_{i=1}^{N}{d_i}\right)\\ 
  & \quad = \sum_{k=1}^{N}{\left\{b_{<k}(a_k-b_k)a_{>k}-d_{<k}(c_k-d_k)c_{>k}\right\}} \\
  & \quad = \sum_{k=1}^{N}{\left\{b_{<k}(a_k-b_k-(c_k-d_k))a_{>k}+(c_k-d_k)(b_{<k}a_{>k}-d_{<k}c_{>k})\right\}}\\
  & \quad = \sum_{k=1}^{N}{b_{<k}(a_k-b_k-(c_k-d_k))a_{>k}}+\sum_{k=1}^{N}{\sum_{j=1}^{N-1}{(c_k-d_k)
      (\beta^k_{<j}(\alpha^k_j-\beta^k_j)\alpha^k_{>j})}},
\end{align*}
and the proof is complete.
\end{proof}

\section{Technical lemmas on fractional Sobolev spaces}
\label{app:b}
We recall a useful lemma about the multiplication of functions in fractional Sobolev spaces, which is a direct consequence of the
Sobolev embedding~\cite[Section 2.1]{Benyi2013a} and of~\cite[Lemma 5, inequality (25)]{Brezis2001a}.

\begin{lemma}
  \label{prod}
  Let $u,v\in H^{s}$, where $s=d/2+\eta$, for some $\eta>0$. Then $uv\in H^{s}$ and there exists
  $K_{\ref{prod}}=K_{\ref{prod}}(d,\eta)$ such that
  \begin{align*}
    \|uv\|_{H^s}\leq K_{\ref{prod}}\|u\|_{H^s}\|v\|_{H^s}.
  \end{align*}
\end{lemma}

The following lemma is an adaptation of the classical multivariate Fa\`a Di Bruno's formula~\cite{constantine1996multivariate} in
the context of weak rather than classical derivatives. We derive it under some restrictive assumptions, which are however
satisfied by the nonlinearity $h_{\delta}$ in our regularised Dean--Kawasaki noise~\eqref{eq:25}.
\begin{lemma}
  \label{faa}
  Let $\alpha\in\{1,\dots,\lfloor d/2\rfloor\}$ and $u\in H^{d/2+\eta}$ for sufficiently small $\eta>0$. Pick $h_{\delta}\in
  C^{\lfloor d/2\rfloor+1}(\mathbb{R})$ with all derivatives up to order $\lfloor d/2\rfloor$ being bounded, and let
  $(x_{\ell_1},\dots,x_{\ell_{\alpha}})$ be an arbitrary element of $\{x_1,\dots,x_d\}^{\alpha}$. Then
  \begin{align}
    \label{i:1}
    \frac{\partial^{(\alpha)}}{\partial x_{\ell_1}\cdots\partial x_{\ell_\alpha}} h_{\delta}(u(\f{x}))
    = \sum_{\pi\in\Pi_{\alpha}}{h^{(|\pi|)}_{\delta}(u(\f{x}))\prod_{j\in J(\pi)}{\prod_{b\in B(\pi)\colon |b|=j}
        {\frac{\partial^{(|b|)}u(\f{x})}{\prod_{z\in b}{\partial x_{\ell_z}}}}}},
  \end{align}
  where we recall the notations $J(\pi)\coloneqq \{j\in\{1,\dots,\alpha \}\colon \beta_j(\pi)>0\}$ and $\beta_j(\pi)\coloneqq
  \#\{b\in B(\pi)\colon |b|=j\}$.  In particular, $\sum_{b\in B(\pi)}{|b|}=\alpha$ for every $\pi\in\Pi_{\alpha}$.
\end{lemma}

\begin{proof}
We only need to show that~\eqref{i:1} holds in the sense of weak derivatives. Fix a test function $\varphi\in
C^{\infty}(\mathbb{T}^d)$. Consider a standard sequence of mollifiers $\varrho_n\colon \mathbb{T}^d\rightarrow \infty$, and set
$u_n\coloneqq \varrho_n\ast u$. As $u_n\in C^{\infty}(\mathbb{T}^d)$, we can apply the classical multivariate Fa\`a Di Bruno's
formula~\cite{constantine1996multivariate} to $h_{\delta}(u_n)$ and perform integration by parts to obtain
\begin{align}
  \label{eq:6004}
  & \int_{\mathbb{T}^d}{\sum_{\pi\in\Pi_{\alpha}}{h^{(|\pi|)}_{\delta}(u_n(\f{x}))\prod_{j\in J(\pi)}{\prod_{b\in B(\pi)\colon |b|=j}
        {\frac{\partial^{(|b|)}u_n(\f{x})}{\prod_{z\in b}{\partial x_{\ell_z}}}}}}\varphi(\f{x})\m\f{x}}\nonumber\\
  & = (-1)^{\alpha}\int_{\mathbb{T}^d}{h_{\delta}(u_n(\f{x}))\frac{\partial^{(\alpha)}}{\partial x_{\ell_1}\cdots\partial x_{\ell_\alpha}}
    \varphi(\f{x})\m\f{x}}.
\end{align}
All we need to do is pass to the limit in~\eqref{eq:6004} to replace $u_n$ with $u$.  Since $u$ is continuous on $\mathbb{T}^d$,
we have $u_n\rightarrow u$ uniformly as $n\rightarrow \infty$. Using the boundedness of $h'_{\delta}$, it is immediate to pass in
the limit in the right-hand side of~\eqref{eq:6004}.  Now fix $\pi\in\Pi_{\alpha}$. The embedding $H^{d/2+\eta-j}\subset
L^{d/(j-\eta)}$ (see~\cite[Corollary 1.2]{Benyi2013a}) implies that, for all blocks $b\in B(\pi)$ with length $j$,
\begin{align*}
{\frac{\partial^{(|b|)}u}{\prod_{z\in b}{\partial x_{\ell_z}}}} \in L^{d/(j-\eta)},
\end{align*}
and, as a result, 
\begin{align*}
  {\frac{\partial^{(|b|)}u_n}{\prod_{z\in b}{\partial x_{\ell_z}}}}=\varrho_n\ast{\frac{\partial^{(|b|)}u}
    {\prod_{z\in b}{\partial x_{\ell_z}}}}\longrightarrow {\frac{\partial^{(|b|)}u}{\prod_{z\in b}{\partial x_{\ell_z}}}}
  \mbox{ in }L^{d/(j-\eta)}\mbox{ as }n\rightarrow \infty.
\end{align*}
We can then settle the convergence of the left-hand side of~\eqref{eq:6004} using the boundedness of derivatives of $h_{\delta}$
and a multi-factor H\"older inequality (for any fixed $\pi\in\Pi_{\alpha}$) on the $|\pi|+2$ terms making up the
product. Specifically, the exponents we use are $q_{h_{\delta}}=q_{\varphi}=\infty$ (for the first and last term), $q_{j}\coloneqq 
d/(j-\eta)$ for the each of the $\beta_j$ terms associated with the product over the set $\{b\in B(\pi)\colon |b|=j\}$, and
$q\coloneqq d/(d-\alpha + \eta\sum_{j\in J}{\beta_j})$ for the remaining (identically unitary) term.
\end{proof}

The following two lemmas are concerned with integrability properties closed related to the product term appearing in the
right-hand side of~\eqref{i:1}.

\begin{lemma}
  \label{der}
  Fix $1\leq \alpha\leq \lfloor d/2 \rfloor$. Fix some partition $\pi\in\Pi_{\alpha}$, and abbreviate $\beta_j(\pi)=\beta_j$ and
  $J(\pi)=J$. Let $\{u_b\}_{b\in B(\pi)}\in [H^{d/2+\eta}]^{|\pi|}$ for some $0<\eta<1/2$. Let
  $(x_{\ell_1},\dots,x_{\ell_{\alpha}})$ be an arbitrary element of $\{x_1,\dots,x_d\}^{\alpha}$.  Then
  \begin{align}
    \label{i:9}
    & \int_{\mathbb{T}^d}{\left\{\prod_{j\in J}{\prod_{b\in B(\pi)\colon |b|=j}{\left|\frac{\partial^{(j)}
            u_b(\f{x})}{\prod_{z\in b}{\partial x_{\ell_z}}}\right|^2}}\right\}\m\f{x}}
    \leq K_{\ref{der}}\prod_{j\in J}{\prod_{b\in B(\pi)\colon |b|=j}{\|u_b\|_{H^{d/2+\eta}}^2}}
  \end{align}
  holds for some $K_{\ref{der}}=K_{\ref{der}}(\pi,d,\eta)>0$.
\end{lemma}

\begin{proof}
We use a multi-factor H\"older inequality on the $|\pi|$ terms making up the product in the left-hand side of~\eqref{i:9}. The
exponents we use are $q_{j}\coloneqq d/(2(j-\eta))$ for the each of the $\beta_j$ terms associated with the product over the set
$\{b\in B(\pi)\colon |b|=j\}$, and $q\coloneqq d/(d-2\alpha + 2\eta\sum_{j\in J}{\beta_j})$ for the remaining (identically
unitary) term. We obtain
\begin{align*}
  \eqref{i:9} & \leq C(\alpha,\pi,d,\eta)\prod_{j\in J}{\prod_{b\in B(\pi)\colon |b|=j}
    {\left\{\int_{\mathbb{T}^d}{\left|\partial^{(j)}u_b(\f{x})/\prod_{z\in b}
        {\partial x_{\ell_z}}\right|^{2q_j}}\m\f{x}\right\}^{1/q_j}}}\nonumber\\
  & \leq K_{\ref{der}}\prod_{j\in J}{\prod_{b\in B(\pi)\colon |b|=j}{\|u_b\|_{H^{d/2+\eta}}^2}},
\end{align*}
where we have used the Sobolev embeddings $H^{d/2+\eta-j}\subset L^{2q_j}$ (see~\cite[Corollary 1.2]{Benyi2013a}) in the final
inequality.
\end{proof}

\begin{lemma}
  \label{gagl}
  Fix some $\pi\in\Pi_{\lfloor d/2 \rfloor}$, and abbreviate $\beta_j(\pi)=\beta_j$ and $J(\pi)=J$. Let $\{u_b\}_{b\in B(\pi)}\in
  [H^{d/2+\eta}]^{|\pi|}$, $v\in H^{d/2+\eta}$ where $0<\eta<C(d)<1/2$ for some small enough $C(d)$. Let
  $(x_{\ell_1},\dots,x_{\ell_{\lfloor d/2 \rfloor}})$ be an arbitrary element of $\{x_1,\dots,x_d\}^{\lfloor d/2 \rfloor}$, and
  let $\f{z}$ be an arbitrary element of $\{\f{x},\f{y}\}^{|\pi|}$.
  \begin{enumerate}
  \item\label{gagl:i} The inequality
    \begin{align}
      \label{i:3}
      & \int_{\mathbb{T}^d}{\int_{\mathbb{T}^d}{\left\{\prod_{j\in J}{\prod_{b\in B(\pi)\colon |b|=j}{\left|\frac{\partial^{(j)}
                u_b(\f{z}_b)}{\prod_{z\in b}{\partial x_{\ell_z}}}\right|^2}}\right\}\frac{\left|v(\f{x})-v(\f{y})\right|^2}
          {|\f{x}-\f{y}|^{d+(d/2-\lfloor d/2\rfloor + \eta )2}}\m\f{x}}\m\f{y}}\nonumber\\
      & \quad \leq K_{\ref{gagl}} \left(\prod_{j\in J}{\prod_{b\in B(\pi)\colon |b|=j}{\|u_b\|_{H^{d/2+\eta}}^2}}\right)
      \left\|v\right\|^{2}_{H^{d/2+\eta}} 
    \end{align}
    holds for some positive $K_{\ref{gagl}}=K_{\ref{gagl}}(d,\eta)$. 
  \item\label{gagl:ii} Pick $\tildej \in J$ and $\tilde{b}\in\{b\in B(\pi):|b|=\tildej \}$. Then we have the inequality
    \begin{align}
      \label{i:6}
      & \int_{\mathbb{T}^d}{\int_{\mathbb{T}^d}{\left\{\prod_{j\in J\setminus \tildej }{\prod_{b\in B(\pi)\colon |b|=j}
            {\left|\frac{\partial^{(j)}u_b(\f{z}_b)}{\prod_{z\in b}{\partial x_{\ell_z}}}\right|^2}}\right\}
          \left\{\prod_{b\in B(\pi)\colon |b|=\tildej ,b\neq \tilde{b}}{\left|\frac{\partial^{(\tildej )}u_b(\f{z}_b)}
            {\prod_{z\in b}{\partial x_{\ell_z}}}\right|^2}\right\}}}\nonumber\\
      & \times \frac{\left|\partial^{(\tildej )}u_{\tilde{b}}(\f{x})/\prod_{z\in \tilde{b}}{\partial x_{\ell_z}}-\partial^{(\tildej )}
        u_{\tilde{b}}(\f{y})/\prod_{z\in \tilde{b}}{\partial x_{\ell_z}}\right|^2}{|\f{x}-\f{y}|^{d+(d/2-\lfloor d/2\rfloor + \eta )2}}
      \m\f{x}\m\f{y}\nonumber\\
      & \quad \leq K_{\ref{gagl}} \left(\prod_{j\in J}{\prod_{b\in B(\pi)\colon |b|=j}{\|u_b\|_{H^{d/2+\eta}}^2}}\right).
    \end{align}
  \end{enumerate}
\end{lemma}

\begin{proof}
It is useful to remember $\sum_{j\in J}{j\beta_j}=\lfloor d/2 \rfloor$.

\emph{Point~\ref{gagl:i}}. We rewrite~\eqref{i:3} as
\begin{align}
  \label{i:4}
  \int_{\mathbb{T}^d}{\int_{\mathbb{T}^d}{\left\{\prod_{j\in J}{\prod_{b\in B(\pi)\colon |b|=j}{\frac{\left|\partial^{(j)}
            u_b(\f{z}_b)/\prod_{z\in b}{\partial x_{\ell_z}}\right|^2}{|\f{x}-\f{y}|^{\gamma\alpha_j\beta^{-1}_j}}}}
      \right\}\frac{\left|v(\f{x})-v(\f{y})\right|^2}{|\f{x}-\f{y}|^{d+2(d/2-\lfloor d/2\rfloor + \eta )-\gamma}}\m\f{x}}\m\f{y}},
\end{align}
for some appropriate $\gamma>0$ and $\{\alpha_j\}_{j\in J}$ such that 
\begin{align}
  \label{c:alpha}
  \alpha_j\in[0,1]\quad\forall j\in J,\quad \sum_{j\in J}{\alpha_j}=1
\end{align} 
to be chosen later. We use a multi-factor H\"older inequality on the $|\pi|+1$ terms making up~\eqref{i:4}. The exponents we use
are $q_{j}\coloneqq d/(2(j-\eta))$ for the each of the $\beta_j$ terms associated with the product over the set $\{b\in
B(\pi)\colon |b|=j\}$, and $q\coloneqq d/(d-2\lfloor d/2\rfloor + 2\eta\sum_{j\in J}{\beta_j})$ for the remaining term. We obtain
\begin{align*}
\eqref{i:4} & \leq \prod_{j\in J}{\prod_{b\in B(\pi)\colon |b|=j}{\left\{\int_{\mathbb{T}^d}{\int_{\mathbb{T}^d}
        {\frac{\left|\partial^{(j)}u_b(\f{z}_b)/\prod_{z\in b}{\partial x_{\ell_z}}\right|^{2q_j}}
          {|\f{x}-\f{y}|^{\gamma\alpha_j\beta^{-1}_j q_j}}}}\m\f{x}\m\f{y}\right\}^{1/q_j}}}\nonumber\\
  & \quad \times \left\{\int_{\mathbb{T}^d}{\int_{\mathbb{T}^d}{\frac{\left|v(\f{x})-v(\f{y})\right|^{2q}}
      {|\f{x}-\f{y}|^{\{d+2(d/2-\lfloor d/2\rfloor + \eta )-\gamma\}q}}}\m\f{x}\m\f{y}}\right\}^{1/q}\nonumber\\
  & = \prod_{j\in J}{\prod_{b\in B(\pi)\colon |b|=j}{\left\{\int_{\mathbb{T}^d}{\frac{1}{|\f{y}|^{\gamma\alpha_j\beta^{-1}_j q_j}}}
      \m\f{y}\right\}^{1/q_j}\left\{\int_{\mathbb{T}^d}{\left|\partial^{(j)}u_b(\f{x})/\prod_{z\in b}{\partial x_{\ell_z}}\right|^{2q_j}}
      \m\f{x}\right\}^{1/q_j}}}\nonumber\\
  & \quad \times \left\{\int_{\mathbb{T}^d}{\int_{\mathbb{T}^d}{\frac{\left|v(\f{x})-v(\f{y})
        \right|^{2q}}{|\f{x}-\f{y}|^{\{d+2(d/2-\lfloor d/2\rfloor + \eta )-\gamma\}q}}}\m\f{x}\m\f{y}}\right\}^{1/q} \\
& \coloneqq \left[\prod_{j\in J}{\prod_{b\in B(\pi)\colon |b|=j}{C^{1/q_j}_j D^{1/q_j}_{j,b}}}\right]\times E^{^{1/q}}.
\end{align*}
We now impose conditions on $\eta$ and $\gamma$ so that $C_j,D_j$, and $E$ are suitably bounded. The integrals $C_j$ may be dealt
with using a standard change of variables in spherical coordinates, and they are bounded if and only if
$-\gamma\alpha_j\beta^{-1}_jq_j+(d-1)>-1$, or equivalently if
\begin{align}
  \label{c:sp}
  \alpha_j < \frac{2\beta_j(j-\eta)}{\gamma},\qquad \forall j\in J.
\end{align}
The terms $D_{j,b}^{1/q_j}$ are bounded, as in the case of Lemma~\ref{der}, by using the Sobolev embedding $H^{d/2+\eta-j}\subset
L^{2q_j}$~\cite[Corollary 1.2]{Benyi2013a}. We now turn to $E$. We rewrite the exponent of $|\f{x}-\f{y}|$ according to the
notation of the space $W^{r,2q}$, for some $r$ to be determined. More precisely, the rewriting
\begin{equation*}
  \{d+2(d/2-\lfloor d/2\rfloor + \eta )-\gamma\}q=d+r(2q)
\end{equation*}
is solved in $r$, giving $r=(d-\gamma)/2+\eta(1-\sum_{j\in J}{\beta_j})$. The restriction $r\in(0,1)$ gives the condition
\begin{align}
  \label{c:s}
  d-2+2\eta(1-\sum_{j\in J}{\beta_j})<\gamma <d+2\eta(1-\sum_{j\in J}{\beta_j}).
\end{align}
The term $E$ may be bounded using the Sobolev embedding $W^{d/2+\eta,2}\subset W^{r,2q}$, and this embedding is true under the
condition~\cite[Theorem 5.1]{Amann2000a}
\begin{align*}
  d/2+\eta-d/2 > r - d/(2q),
\end{align*}
which is equivalent to
\begin{align}
  \label{c:E}
  \gamma > 2\lfloor d/2\rfloor-4\eta\sum_{j\in J}\beta_j.
\end{align}
If we pick $\gamma\coloneqq 2\lfloor d/2\rfloor-3\eta\sum_{j\in J}\beta_j$ and $\eta$ small enough, then~\eqref{c:E}
and~\eqref{c:s} are satisfied. Furthermore, summing the right-hand side of~\eqref{c:sp} over $j$, we obtain
\begin{align}
  \sum_{j\in J}{\frac{2\beta_j(j-\eta)}{2\lfloor d/2\rfloor-3\eta\sum_{j\in J}\beta_j}}
  =\frac{2\lfloor d/2\rfloor-2\eta\sum_{j\in J}{\beta_j}}{2\lfloor d/2\rfloor-3\eta\sum_{j\in J}{\beta_j}}>1.
\end{align}
The above inequality implies that the $\alpha_j$'s can be chosen so that~\eqref{c:alpha} and~\eqref{c:sp} are satisfied. As a
result of the bounds for $C_j,D_{j,b},E$, the inequality~\eqref{i:3} follows and \emph{Point~\ref{gagl:i}} is settled.

\emph{Point~\ref{gagl:ii}}. The case $\sum_{j\in J}{\beta_j}=1$ uniquely corresponds to having $\tildej =\lfloor d/2 \rfloor $
and $\beta_{\tildej }=1$. Therefore, the only term surviving in the product of integrands in the left-hand side of~\eqref{i:6} is
the last term, and the result is trivial.

We consider all the other cases, where necessarily $\sum_{j\in J}{\beta_j}>1$. We rewrite~\eqref{i:6} as
\begin{align}
  \label{i:7}
  & \int_{\mathbb{T}^d}{\int_{\mathbb{T}^d}{\left\{\prod_{j\in J\setminus \tildej }{\prod_{b\in B(\pi)\colon |b|=j}
        {\frac{\left|\partial^{(j)}u_b(\f{z}_b)/\prod_{z\in b}{\partial x_{\ell_z}}\right|^2}
          {|\f{x}-\f{y}|^{\gamma\alpha_j\beta^{-1}_j}}}}\right\}}}\nonumber\\
  & \quad\times\left\{\prod_{b\in B(\pi)\colon |b|=\tildej ,b\neq \tilde{b}}{\frac{\left|\partial^{(\tildej )}
      u_b(\f{z}_b)/\prod_{z\in b}{\partial x_{\ell_z}}\right|^2}{|\f{x}-\f{y}|^{\gamma\alpha_{\tildej }(\beta_{\tildej }-1)^{-1}}}}
  \right\}\nonumber\\
  & \quad \times \frac{\left|\partial^{(\tildej )}u_{\tilde{b}}(\f{x})/\prod_{z\in \tilde{b}}{\partial x_{\ell_z}}-\partial^{(\tildej )}
    u_{\tilde{b}}(\f{y})/\prod_{z\in \tilde{b}}{\partial x_{\ell_z}}\right|^2}{|\f{x}-\f{y}|^{d+(d/2-\lfloor d/2\rfloor + \eta )2-\gamma}}
  \m\f{x}\m\f{y},
\end{align}
where the second curly brackets is understood to be equal to $1$ should $\beta_{\tildej }=1$, for some appropriate $\gamma>0$ and
$\{\alpha_j\}_{j\in J^{\star}}$ such that
\begin{align}
  \label{c:alpha2}
  \alpha_j\in[0,1]\quad\forall j\in J^{\star},\quad \sum_{j\in J^{\star}}{\alpha_j}=1
\end{align} 
to be chosen later, where $J^{\star}=J$ if $\beta_{\tildej }>1$, and $J^{\star}=J\setminus \tildej $ otherwise. We use a
multi-factor H\"older inequality on the $|\pi|$ terms making up~\eqref{i:7}. The exponents we use are $q_{j}\coloneqq
d/(2(j-\eta))$ for the each of the $\beta_j$ terms associated with the product over the set $\{b\in B(\pi)\colon |b|=j\}$ for
$j\in J\setminus \tildej $, then $q_{\tildej }\coloneqq d/(2(\tildej -\eta))$ for the each of the $\beta_{\tildej }-1$ terms
associated with the product over the set $\{b\in B(\pi)\colon |b|=\tildej ,b\neq \tilde{b}\}$, and finally $q\coloneqq
d/(d-2\lfloor d/2\rfloor + 2\eta\sum_{j\in J}{\beta_j}+2(\tildej -\eta))$ for the remaining term. 
We obtain
\begin{align*}%
\eqref{i:7} 
& \leq \prod_{j\in J\setminus\tildej }{\prod_{b\in B(\pi)\colon |b|=j}{\left\{\int_{\mathbb{T}^d}{\frac{1}
      {|\f{y}|^{\gamma\alpha_j\beta^{-1}_jq_j}}}\m\f{y}\right\}^{1/q_j}\left\{\int_{\mathbb{T}^d}{\left|\partial^{(j)}
      u_b(\f{x})/\prod_{z\in b}{\partial x_{\ell_z}}\right|^{2q_j}}\m\f{x}\right\}^{1/q_j}}}\nonumber\\
& \quad \times \prod_{b\in B(\pi)\colon |b|=\tildej ,b\neq \tilde{b}}{\left\{\int_{\mathbb{T}^d}{\frac{1}
    {|\f{y}|^{\gamma\alpha_{\tildej }(\beta_{\tildej }-1)^{-1}q_{\tildej }}}}\m\f{y}\right\}^{1/q_{\tildej }}\left\{\int_{\mathbb{T}^d}
  {\left|\partial^{(j)}u_b(\f{x})/\prod_{z\in b}{\partial x_{\ell_z}}\right|^{2q_{\tildej }}}\m\f{x}\right\}^{1/q_{\tildej }}}\nonumber\\
& \quad \times \left\{\int_{\mathbb{T}^d}{\int_{\mathbb{T}^d}{\frac{\left|\partial^{(\tildej )}u_{\tilde{b}}(\f{x})/
      \prod_{z\in \tilde{b}}{\partial x_{\ell_z}}-\partial^{(\tildej )}u_{\tilde{b}}(\f{y})/\prod_{z\in \tilde{b}}{\partial x_{\ell_z}}
      \right|^{2q}}{|\f{x}-\f{y}|^{\{d+2(d/2-\lfloor d/2\rfloor + \eta )-\gamma\}q}}}\m\f{x}\m\f{y}}\right\}^{1/q}\\
&  \coloneqq \left[\prod_{j\in J\setminus\tildej }{\prod_{b\in B(\pi)\colon |b|=j}{C^{1/q_j}_jD^{1/q_j}_{j,b}}}\right]\times
\left[\prod_{b\in B(\pi)\colon |b|=j,b\neq\tilde{b}}{C^{1/q_{\tildej }}_{\tildej }D^{1/q_{\tildej }}_{{\tildej },b}}\right]\times E^{^{1/q}}.
\end{align*}
Bounding the above involves similar discussions as per \emph{Point~\ref{gagl:i}}. More specifically, the boundedness of the
spherical integrals (the $C_j$'s above) is granted under the conditions
\begin{align}
  \label{c:sp2}
  \alpha_j < \frac{2\beta_j(j-\eta)}{\gamma},\qquad \forall j\in J^{\star}\setminus\tildej ,\qquad \alpha_{\tildej }<
  \frac{2(\beta_{\tildej }-1)(\tildej -\eta)}{\gamma},
\end{align}
with the last condition only imposed if $\tildej \in J^{\star}$. The bound for the terms $D_{j,b}^{1/q_j}$ is settled exactly as
in \emph{Point~\ref{gagl:i}}. As for $E$, we solve the equation
\begin{equation*}
  \{d+2(d/2-\lfloor d/2\rfloor + \eta )-\gamma\}q = d + r(2q)
\end{equation*}
in the variable $r$, thus getting $r\coloneqq (d-\gamma)/2+\eta(1-\sum_{j\in J}{\beta_j})-(\tildej -\eta)$. The constraint
$r\in(0,1)$ results in the requirement
\begin{align}
  \label{c:s2}
  d-2+2\eta(1-\sum_{j\in J}{\beta_j})-2(\tildej -\eta)<\gamma <d+2\eta(1-\sum_{j\in J}{\beta_j})-2(\tildej -\eta).
\end{align}
We control $E$ using the embedding $H^{d/2+\eta}\subset W^{\tildej +r,2q}$, which is valid under the constraint
\begin{align*}
  d/2+\eta-d/2 > \tildej +r - d/(2q),
\end{align*}
which is equivalent to
\begin{align}
  \label{c:E2}
  \gamma > 2\lfloor d/2\rfloor+4\eta(1-\sum_{j\in J}\beta_j)-2\tildej .
\end{align}
If we take $\gamma\coloneqq 2\lfloor d/2\rfloor+3\eta(1-\sum_{j\in J}\beta_j)-2\tildej $ and $\eta$ small enough,
then~\eqref{c:s2} and~\eqref{c:E2} are satisfied. Furthermore, summing all the right-hand sides in~\eqref{c:sp2} gives
\begin{align*}
  & \mathbf{1}_{\beta_{\tildej }>1}\frac{2(\beta_{\tildej }-1)(\tildej -\eta)}{2\lfloor d/2\rfloor+3\eta(1-\sum_{j\in J}\beta_j)
    -2\tildej }+\sum_{j\in J\setminus \tildej }{\frac{2\beta_j(j-\eta)}{2\lfloor d/2\rfloor+3\eta(1-\sum_{j\in J}\beta_j)
      -2\tildej }} \nonumber\\
  & \quad = \sum_{j\in J}{\frac{2\beta_j(j-\eta)}{2\lfloor d/2\rfloor+3\eta(1-\sum_{j\in J}\beta_j)-2\tildej }}
  - \frac{2(\tildej -\eta)}{2\lfloor d/2\rfloor+3\eta(1-\sum_{j\in J}\beta_j)-2\tildej }\nonumber\\
  & \quad = \frac{2\lfloor d/2\rfloor+2\eta(1-\sum_{j\in J}\beta_j)-2\tildej }{2\lfloor d/2\rfloor
    +3\eta(1-\sum_{j\in J}\beta_j)-2\tildej }>1,
\end{align*}
where the last inequality is valid because $\sum_{j\in J}{\beta_j}>1$. Therefore the $\alpha_j$'s can be chosen so
that~\eqref{c:alpha2} and~\eqref{c:sp2} are satisfied. As a result of the bounds for $C_j,D_{j,b},E$, the inequality~\eqref{i:6}
follows and \emph{Point~\ref{gagl:ii}} is settled.
\end{proof}

{\bfseries Acknowledgements} 
All authors thank the anonymous referee for his/her careful reading of the manuscript and valuable suggestions.
This paper was motivated by stimulating discussions at the First Berlin--Leipzig Workshop on
Fluctuating Hydrodynamics in August 2019 with Ana Djurdjevac, Rupert Klein and Ralf Kornhuber. JZ gratefully acknowledges funding
by a Royal Society Wolfson Research Merit Award. FC gratefully acknowledges funding from the European Union's Horizon 2020
research and innovation programme under the Marie Sk\l{}odowska-Curie grant agreement No. 754411.

\def\cprime{$'$} \def\cprime{$'$} \def\cprime{$'$}
  \def\polhk#1{\setbox0=\hbox{#1}{\ooalign{\hidewidth
  \lower1.5ex\hbox{`}\hidewidth\crcr\unhbox0}}} \def\cprime{$'$}
  \def\cprime{$'$}

\def\cprime{$'$} \def\cprime{$'$} \def\cprime{$'$}
  \def\polhk#1{\setbox0=\hbox{#1}{\ooalign{\hidewidth
  \lower1.5ex\hbox{`}\hidewidth\crcr\unhbox0}}} \def\cprime{$'$}
  \def\cprime{$'$}


\begin{thebibliography}{21}
\providecommand{\natexlab}[1]{#1}
\providecommand{\url}[1]{\texttt{#1}}
\expandafter\ifx\csname urlstyle\endcsname\relax
  \providecommand{\doi}[1]{doi: #1}\else
  \providecommand{\doi}{doi: \begingroup \urlstyle{rm}\Url}\fi

\bibitem[Abramowitz and Stegun(1964)]{Abramowitz1964a}
Milton Abramowitz and Irene~A. Stegun.
\newblock \emph{Handbook of mathematical functions with formulas, graphs, and
  mathematical tables}, volume~55 of \emph{National Bureau of Standards Applied
  Mathematics Series}.
\newblock For sale by the Superintendent of Documents, U.S. Government Printing
  Office, Washington, D.C., 1964.

\bibitem[Amann(2000)]{Amann2000a}
Herbert Amann.
\newblock Compact embeddings of vector-valued {S}obolev and {B}esov spaces.
\newblock \emph{Glas. Mat. Ser. III}, 35(55)\penalty0 (1):\penalty0 161--177,
  2000.
\newblock ISSN 0017-095X.
\newblock URL \url{https://hrcak.srce.hr/file/7912}.
\newblock Dedicated to the memory of Branko Najman.

\bibitem[B\'{e}nyi and Oh(2013)]{Benyi2013a}
\'{A}rp\'{a}d B\'{e}nyi and Tadahiro Oh.
\newblock The {S}obolev inequality on the torus revisited.
\newblock \emph{Publ. Math. Debrecen}, 83\penalty0 (3):\penalty0 359--374,
  2013.
\newblock ISSN 0033-3883.
\newblock \doi{10.5486/PMD.2013.5529}.

\bibitem[Brezis and Mironescu(2001)]{Brezis2001a}
Ha\"{\i}m Brezis and Petru Mironescu.
\newblock Gagliardo-{N}irenberg, composition and products in fractional
  {S}obolev spaces.
\newblock \emph{J. Evol. Equ.}, 1\penalty0 (4):\penalty0 387--404, 2001.
\newblock ISSN 1424-3199.
\newblock \doi{10.1007/PL00001378}.
\newblock Dedicated to the memory of Tosio Kato.

\bibitem[Cates and Tailleur(2015)]{Cates2015a}
Michael~E. Cates and Julien Tailleur.
\newblock Motility-induced phase separation.
\newblock \emph{Annu. Rev. Condens. Matter Phys.}, 6\penalty0 (1):\penalty0
  219--244, 2015. 
\newblock \doi{10.1146/annurev-conmatphys-031214-014710}.


\bibitem{constantine1996multivariate}
Gregory M. Constantine and Thomas H. Savits.
\newblock A multivariate Faa di Bruno formula with applications.
\newblock \emph{Trans. Amer. Math. Soc.}, 348\penalty0 (2):\penalty0 503--520,
  1996.
\newblock \doi{10.1090/S0002-9947-96-01501-2}.


\bibitem[Cornalba et~al.(2019)Cornalba, Shardlow, and Zimmer]{Cornalba2019a}
Federico Cornalba, Tony Shardlow, and Johannes Zimmer.
\newblock A regularised {D}ean--{K}awasaki model: derivation and analysis.
\newblock \emph{SIAM J. Math. Anal.}, 51\penalty0 (2):\penalty0 1137--1187,
  2019.
\newblock ISSN 0036-1410.
\newblock \doi{10.1137/18M1172697}.


\bibitem[Cornalba et~al.(2020)Cornalba, Shardlow, and Zimmer]{Cornalba2020a}
Federico Cornalba, Tony Shardlow, and Johannes Zimmer.
\newblock From weakly interacting particles to a regularised {D}ean-{K}awasaki
  model.
\newblock \emph{Nonlinearity}, 33\penalty0 (2):\penalty0 864--891, 2020.
\newblock ISSN 0951-7715.
\newblock \doi{10.1088/1361-6544/ab5174}.

\bibitem[Da~Prato and Zabczyk(2014)]{Da-Prato2014a}
Giuseppe Da~Prato and Jerzy Zabczyk.
\newblock \emph{Stochastic equations in infinite dimensions}, volume 152 of
  \emph{Encyclopedia of Mathematics and its Applications}.
\newblock Cambridge University Press, Cambridge, second edition, 2014.
\newblock ISBN 978-1-107-05584-1.
\newblock \doi{10.1017/CBO9781107295513}.

\bibitem[Dean(1996)]{Dean1996a}
David~S. Dean.
\newblock Langevin equation for the density of a system of interacting
  {L}angevin processes.
\newblock \emph{J. Phys. A}, 29\penalty0 (24):\penalty0 L613, 1996.
\newblock \doi{10.1088/0305-4470/29/24/001}.

\bibitem[Donev et~al.(2014)Donev, Fai, and Vanden-Eijnden]{Donev2014a}
Aleksandar Donev, Thomas~G Fai, and Eric Vanden-Eijnden.
\newblock A reversible mesoscopic model of diffusion in liquids: from giant
  fluctuations to fick's law.
\newblock \emph{J. Stat. Mech: Theory Exp.}, 2014\penalty0 (4):\penalty0
  P04004, 2014.
\newblock \doi{10.1088/1742-5468/2014/04/p04004}.

\bibitem[Duong and Pavliotis(2018)]{Duong2018b}
Manh~Hong Duong and Grigorios~A. Pavliotis.
\newblock Mean field limits for non-{M}arkovian interacting particles:
  convergence to equilibrium, generic formalism, asymptotic limits and phase
  transitions.
\newblock \emph{Commun. Math. Sci.}, 16\penalty0 (8):\penalty0 2199--2230,
  2018.
\newblock ISSN 1539-6746.
\newblock \doi{10.4310/CMS.2018.v16.n8.a7}.


\bibitem[Fehrman and Gess(2019)]{Fehrman2019a}
Benjamin Fehrman and Benjamin Gess.
\newblock Well-posedness of nonlinear diffusion equations with nonlinear,
  conservative noise.
\newblock \emph{Arch. Ration. Mech. Anal.}, 233\penalty0 (1):\penalty0
  249--322, 2019.
\newblock ISSN 0003-9527.
\newblock \doi{10.1007/s00205-019-01357-w}.


\bibitem[Gess et~al.(2016)Gess, Perthame, and Souganidis]{Gess2016a}
Benjamin Gess, Beno\^{\i}t Perthame, and Panagiotis~E. Souganidis.
\newblock Semi-discretization for stochastic scalar conservation laws with
  multiple rough fluxes.
\newblock \emph{SIAM J. Numer. Anal.}, 54\penalty0 (4):\penalty0 2187--2209,
  2016.
\newblock ISSN 0036-1429.
\newblock \doi{10.1137/15M1053670}.

\bibitem[Gradshteyn and Ryzhik(2007)]{Gradshteyn2007a}
I.~S. Gradshteyn and I.~M. Ryzhik.
\newblock \emph{Table of integrals, series, and products}.
\newblock Elsevier/Academic Press, Amsterdam, seventh edition, 2007.
\newblock ISBN 978-0-12-373637-6; 0-12-373637-4.
\newblock Translated from the Russian, Translation edited and with a preface by
  Alan Jeffrey and Daniel Zwillinger, With one CD-ROM (Windows, Macintosh and
  UNIX).

\bibitem[Helfmann et~al.(2019)Helfmann, Djurdjevac~Conrad, Djurdjevac,
  Winkelmann, and Sch\"utte]{Helfmann2019a}
Luzie Helfmann, Natasa Djurdjevac~Conrad, Ana Djurdjevac, Stefanie Winkelmann,
  and Christof Sch\"utte.
\newblock From interacting agents to density-based modeling with stochastic
  {PDE}s.
  \newblock \emph{arXiv preprint arXiv:1905.13525}, 2019.
\newblock \url{https://arxiv.org/abs/1905.13525}.


\bibitem[Kawasaki(1998)]{Kawasaki1998a}
Kyozi Kawasaki.
\newblock Microscopic analyses of the dynamical density functional equation of
  dense fluids.
\newblock \emph{J. Statist. Phys.}, 93\penalty0 (3-4):\penalty0 527--546, 1998.
\newblock \doi{10.1023/B:JOSS.0000033240.66359.6c}.


\bibitem[Konarovskyi et~al.(2019)Konarovskyi, Lehmann, and von
  Renesse]{Konarovskyi2019a}
Vitalii Konarovskyi, Tobias Lehmann, and Max-K. von Renesse.
\newblock Dean--{Kawasaki} dynamics: ill-posedness vs.\ triviality.
\newblock \emph{Electron. Commun. Probab.}, 24:\penalty0 Paper No. 8, 9, 2019.
\newblock ISSN 1083-589X.
\newblock \doi{10.1214/19-ECP208}.


\bibitem[Lutsko(2012)]{Lutsko2012a}
James~F. Lutsko.
\newblock A dynamical theory of nucleation for colloids and macromolecules.
\newblock \emph{J. Chem. Phys.}, 136:\penalty0 034509, 2012.
\newblock \doi{10.1063/1.3677191}.


\bibitem[N{\aa}sell(1978)]{Nasell1978a}
Ingemar N{\aa}sell.
\newblock Rational bounds for ratios of modified {B}essel functions.
\newblock \emph{SIAM J. Math. Anal.}, 9\penalty0 (1):\penalty0 1--11, 1978.
\newblock ISSN 0036-1410.
\newblock \doi{10.1137/0509001}.

\bibitem[Pazy(1983)]{Pazy1983a}
A.~Pazy.
\newblock \emph{Semigroups of linear operators and applications to partial
  differential equations}, volume~44 of \emph{Applied Mathematical Sciences}.
\newblock Springer-Verlag, New York, 1983.
\newblock ISBN 0-387-90845-5.
\newblock \doi{10.1007/978-1-4612-5561-1}.

\bibitem[Rotskoff and Vanden-Eijnden(2018)]{Rotskoff2018a}
Grant~M. Rotskoff and Eric Vanden-Eijnden.
\newblock Trainability and accuracy of neural networks: An interacting particle
  system approach.
    \newblock \emph{arXiv preprint arXiv:1805.00915}, 2018.
\newblock \url{https://arxiv.org/abs/1805.00915}.

\bibitem[Tappe(2012)]{Tappe2012a}
Stefan Tappe.
\newblock Some refinements of existence results for {SPDE}s driven by {W}iener
  processes and {P}oisson random measures.
\newblock \emph{Int. J. Stoch. Anal.}, pages Art. ID 236327, 24, 2012.
\newblock ISSN 2090-3332.
\newblock \doi{10.1155/2012/236327}.

\bibitem[Thompson et~al.(2011)Thompson, Tailleur, Cates, and
  Blythe]{Thompson2011a}
A.~G. Thompson, J.~Tailleur, M.~E. Cates, and R.~A. Blythe.
\newblock Lattice models of nonequilibrium bacterial dynamics.
\newblock \emph{J. Stat. Mech. Theory Exp.}, \penalty0 (2):\penalty0 P02029,
  34, 2011.
  \newblock \doi{10.1088/1742-5468/2011/02/P02029}.

\end{thebibliography}
\end{document}